\documentclass[12pt,reqno]{amsart}
\usepackage{amssymb,comment}
\usepackage{hyperref}
\usepackage{stmaryrd}
\usepackage{slashed}
\usepackage{url}
\usepackage{graphicx}

\newcounter{count}[section]
\renewcommand{\thecount}{\arabic{section}.\arabic{count}}

\numberwithin{equation}{section} 

\newenvironment{Umgeb1rekursiv}[1]
   {\vspace{0.5cm}
     \noindent
     \par\noindent
     \refstepcounter{count}
     \textbf{#1~\thecount}
     \hspace{0.2cm}
     \itshape
   }
   {\vspace{0.2cm}\par}

\newenvironment{Umgeb1}[1]
   {\vspace{0.5cm}
     \par\noindent
     \refstepcounter{count}
     \textbf{#1~\thecount}
     \hspace{0.2cm}
   }
   {\vspace{0.2cm}\par}

\newenvironment{example}{\begin{Umgeb1}{Example}}
   {\end{Umgeb1}}
\newenvironment{prop}{\begin{Umgeb1rekursiv}{Proposition}}
  {\end{Umgeb1rekursiv}}
\newenvironment{lem}{\begin{Umgeb1rekursiv}{Lemma}}
  {\end{Umgeb1rekursiv}}
\newenvironment{corollary}{\begin{Umgeb1rekursiv}{Corollary}}
  {\end{Umgeb1rekursiv}}  
\newenvironment{theorem}{\begin{Umgeb1rekursiv}{Theorem}}
  {\end{Umgeb1rekursiv}}

\newenvironment{remark}{\begin{Umgeb1rekursiv}{Remark}}
   {\end{Umgeb1rekursiv}}
   
\newcommand{\R}{\mathbb{R}}

\newcommand{\C}{\mathbb{C}}
\newcommand{\N}{\mathbb{N}}

\newcommand{\id}{\operatorname{Id}}

\newcommand{\dm}{d}
\def\End{\operatorname{End}}

\newcommand{\Res}{\operatorname{Res}}

\newcommand{\abs}[1]{\lvert#1\rvert}
\renewcommand{\part}[2]{\frac{\partial #1}{\partial #2}}

\newcommand{\sgn}{sgn}

\def\st{\stackrel{\text{def}}{=}}

\title{Bernstein-Sato identities and conformal symmetry breaking operators}
\author[Matthias Fischmann, Bent {\O}rsted and Petr Somberg]{Matthias Fischmann, Bent {\O}rsted and Petr Somberg}

\allowdisplaybreaks
\begin{document}
\address{Department of Mathematics of Aarhus University, Ny Munkegade 118, 8000 Aarhus C, Denmark and 
Mathematical Institute of MFF UK, Sokolovsk\'{a} 83, 18000 Praha 8, Czech Republic}

\thanks{M. Fischmann and B. {\O}rsted were supported by the Department of Mathematics (Aarhus University) and the Danish Research Council, P. Somberg was supported by GA P201/12/G028.}

\email{fischmann@math.au.dk, orsted@math.au.dk, somberg@karlin.mff.cuni.cz}



\begin{abstract}
 We present Bernstein-Sato identities for scalar-, spinor- and differential form-valued distribution kernels
on Euclidean space associated 
 to conformal symmetry breaking operators. 
 The associated Bernstein-Sato operators lead to partially new formulae for conformal symmetry breaking differential 
 operators on functions, spinors and differential forms.
\end{abstract}
\maketitle

\tableofcontents

\section{Introduction}
Many aspects of harmonic analysis on Euclidean $n$-space, the $n$-sphere and 
$n$-hyperbolic space
are related to the action of the conformal group. This is true not only for
functions, but also for spinors and differential forms. Motivated partly by
representation theory and partly by conformal differential geometry, there has
recently been much progress in establishing concrete and explicit formulas
for natural integral and differential operators exhibiting some form of
conformal invariance. 
%Here a key concept is that of 
%T. Kobayashi's symmetry-breaking
%operators \cite{KS} in representation theory; \cite{J1} treats analogous families
%in the framework of conformal differential geometry.
Here a key concept is that of 
  A. Juhl's residue families \cite{J1} in the framework of conformal geometry and 
	T. Kobayashi's symmetry-breaking
  operators \cite{KS} in representation theory.

In this paper we collect and extend many formulas related to distribution kernels
for both integral and differential operators in the three basic situations
of functions, spinors, and differential forms. We shall treat both the absolute
case of the conformal group, as well as the relative case of the conformal
groups of Euclidean space and a coordinate hyperplane, i.e. the case 
of conformal symmetry-breaking
operators. As it turns out, there is a series of natural identities based
on some partly new second-order operators which are termed in the present 
article {\it Bernstein-Sato operators}.

Among the very useful and important ingredients in the theory of meromorphic 
 continuation of families of distributions and the closely related theory of 
 ${\fam2 D}$-modules are the Bernstein-Sato identities, see the seminal paper 
 \cite{Bernstein} by J. Bernstein. In particular, they allow 
 to find the precise position of the corresponding poles as well as 
 to encode a recurrence structure for the residues in the distribution family under consideration. 
 
In recent years there appeared several approaches to a classification scheme for
conformally covariant differential operators $P_{2N}, \slashed{D}_{2N+1}$ and $L^{(p)}_{2N}$ 
 (acting on functions, spinors and differential $p$-forms) 
 on semi-Riemannian (spin-)manifolds, cf. \cite{GJMS, GMP1, BransonGover}. 
Furthermore, the operators $P_{2N}$ and $\slashed{D}_{2N+1}$ were extended to a 
 theory of $1$-parameter families of conformally covariant differential 
 operators \cite{J1,FS}, nowdays known and termed as the {\it residue families}. 
These correspond to the relative case. Concerning 
 differential forms, not much is so far known and available in the literature. 

 The above mentioned operators ($P_{2N}, \slashed{D}_{2N+1}$ and $L^{(p)}_{2N}$) on Euclidean space $\R^n$ arise as residues of Knapp-Stein 
intertwining integral operators for certain families of induced representations of
conformal Lie groups, cf. \cite{KnappStein}.  
These operators are $1$-parameter families of pseudo-differential convolution operators with 
respect to Riesz distributions on functions (\cite{Riesz}), spinors (\cite{CO}) and differential forms 
(\cite{FO}), respectively. 

 One may also study Knapp-Stein type operators associated to a pair of conformal Lie groups
(the relative case); these form a $2$-parameter 
family of distributions, see \cite{KS, MO, K2}. 
They are termed 
{\it conformal symmetry breaking operators}, and are intertwining integral operators
acting on principal series representations (realized for example in the
non-compact picture) for the action of conformal Lie algebras
on $\R^n$ and $\R^{n-1}$, respectively. Their residues are 
given by $1$-parameter families of equivariant differential operators termed {\it conformal symmetry breaking differential operators}.
Note that these conformal symmetry breaking differential operators are just specific cases of residue family operators. The conformal symmetry breaking operators were studied in a more general context, e.g. the case
of conformal Lie groups for $\R^n$ and $\R^{n-m}$ with $1\leq m\leq n-1$ is discussed in \cite{MO1}. 
However, curved generalizations of the residue families, both of co-dimension 
one, and of higher co-dimensions, are not yet 
properly understood. One of the motivations for the present paper is to gain a better understanding
of the model (flat) case in order to undertake a subsequent study of the curved generalizations, related
to AdS/CFT and the corresponding Poincare-Einstein geometry.
 
 Thus in the present paper, we prove Bernstein-Sato identities for distribution kernels of the three basic types of 
 conformal symmetry breaking operators: scalar-valued (acting on $\mathcal{D}^\prime(\R^n)$)
 \begin{align}
   K^+_{\lambda,\nu}(x^\prime,x_n)&= \abs{x_n}^{\lambda+\nu-n}(\abs{x^\prime}^2+x_n^2)^{-\nu},\notag\\
   K^-_{\lambda,\nu}(x^\prime,x_n)&= \sgn(x_n)\abs{x_n}^{\lambda+\nu-n}(\abs{x^\prime}^2+x_n^2)^{-\nu},
 \end{align}
 spinor-valued (acting on $\slashed{\mathcal{D}}^\prime(\R^n)$)
  \begin{align}
   \slashed{K}^\pm_{\lambda,\nu}(x^\prime,x_n)&= K^\pm_{\lambda-\frac 12,\nu+\frac 12}(x^\prime,x_n) x\cdot,
 \end{align}
and differential form-valued (acting on $\mathcal{D}^{\prime, p}(\R^n)$)
 \begin{align}
   K^{(p),\pm}_{\lambda,\nu}(x^\prime,x_n)= K^\pm_{\lambda-1,\nu+1}(x^\prime,x_n)(i_x\varepsilon_x-\varepsilon_x i_x)i_{e_n}\varepsilon_{e_n}.
 \end{align}
 We use the notation $\lambda,\nu\in\C$, $x=(x^\prime,x_n)\in \R^n$ with $x^\prime\in \R^{n-1}$, 
 $x\cdot$ for the Clifford multiplication with the vector $x\in\R^n$, and $i_x,\varepsilon_x$ are 
the interior and exterior products with respect to $x$ on differential forms. 
 Our convention for the distribution kernels are as duals to what they are 
considered as in the standard 
literature mentioned above. This choice is justified by the use of Fourier transform, which is applied to 
 these distribution kernels directly without further dualization and leads to generalizations 
of singular vectors studied in \cite{KOSS, FJS, KKP}.
  
 The proposal to find Bernstein-Sato operators of interest in the context of conformal 
 symmetry breaking operators was initiated in \cite{PS}, where certain shift operators 
 for Gegenbauer polynomials regarded as the residues of Fourier transformed 
 $K^\pm_{\lambda,\nu}(x^\prime,x_n)$ are studied. Later on, a more sophisticated approach was 
 suggested in a private communication by J.L. Clerc. Moreover, the 
 Bernstein-Sato operators are themselves intertwining operators for the relevant 
conformal Lie groups. Concerning detailed notation 
and intertwining results we refer to \cite{C1,C}. 
 
 Our paper is structured as follows. In Section \ref{Preliminaries}, we fix the 
 notation and recall fairly standard results related to Riesz distributions 
 on functions, spinors and differential forms. 
 
 In Section \ref{BSOperatorsAndIdentities}, we present a construction of 
Bernstein-Sato type operators for functions
 $P(\lambda)$, \eqref{eq:BSOperator1}, spinors $\slashed{P}(\lambda)$, 
\eqref{eq:BSOperatorSpinor1}, and differential forms $P^p(\lambda)$, 
 \eqref{eq:BSOperatorForDiffForms}. Furthermore, 
 we show that they satisfy a Bernstein-Sato identity on the space of 
 distribution kernels for functions in Theorem \ref{BSGeneralizedRiesz1}, 
 spinors in Theorem \ref{BSGeneralizedRiesz1Spinor}, and differential forms
in Theorem \ref{BSIdentityForDiffForms},  respectively.
 
 In Section \ref{ComAndApps}, we comment on the origin 
 of the constructed Bernstein-Sato operators (it is interesting that they
have several, a priori quite different definitions) 
 and discuss some applications to the conformal symmetry breaking differential 
 operators on functions, spinors and differential forms. 
 The construction results in {\it known} formulas for conformal symmetry breaking differential 
 operators on functions, cf. Theorem \ref{BSFamily}, and {\it new} formulas for 
 conformal symmetry breaking differential operators on spinors and differential forms, 
 see Theorems \ref{BSVsSBOSpinor} and \ref{BSVsSBODiffForms}. 
 
 {\bf Acknowlegment:} We would like to express our thanks to J.L. Clerc for the private 
discussion leading to the present paper.

%%%%%%%%%%%%%%%%%%%%%%%%%%%%%%%%%%%%%%%%%%%%%%%%%%%%%%%%%%%%%%%%%%%%%%%%%%%%%%%%%%%%%%%%%
\section{Preliminaries}\label{Preliminaries}
  Let $\R^n$ be equipped with the canonical 
	flat metric $\langle\cdot,\cdot\rangle$. We collect some basic 
	known facts concerning 
  tree types of Riesz distributions: for scalars \cite{Riesz,GelfandShilov}, spinors \cite{CO} 
  and differential forms \cite{FO}.
  
  We denote by $\mathcal{S}(\R^n)$ the algebra of Schwartz functions on $\R^n$ and 
  follow the convention for the Fourier transform 
  \cite{GelfandShilov} on Schwartz functions $f\in\mathcal{S}(\R^n)$:
  \begin{align*}
    \mathcal{F}(f)(\xi)\st \int_{\R^n} f(x)e^{i \langle x, \xi\rangle}dx,
  \end{align*}
  which also extends to the space of tempered distributions $\mathcal{S}^\prime(\R^n)$. 
  Note the identity
  \begin{align*}
  \mathcal{F}(f\star g)(\xi)=\mathcal{F}(f)(\xi)\mathcal{F}(g)(\xi),
  \end{align*}
  where $(f\star g)(x)\st \int_{\R^n}f(x-y)g(y)dy$ denotes the convolution of Schwartz functions $f$ and $g$. 
  This normalization of the Fourier transform is chosen in such a way that 
  $\mathcal{F}(\delta_0)=1$, where $\delta_0$ is the Dirac distribution centered at the origin. 
  Recall that for a polynomial $P$ in $n$ variables we have the identities
  \begin{align}
    P(\partial_{\xi_1},\ldots,\partial_{\xi_n})\mathcal{F}(f)(\xi)
      &=\mathcal{F}(P(i x_1,\ldots,i x_n)f)(\xi),\notag\\
    \mathcal{F}(P(\partial_{x_1},\ldots,\partial_{x_n})f)(\xi)
      &=P(-i \xi_1,\ldots,-i \xi_n)\mathcal{F}(f)(\xi)\label{eq:FourierProperties2}
  \end{align}
  for $f\in\mathcal{S}(\R^n)$. 
  
  The Fourier transform $\mathcal{F}$ extends to the space of spinor-valued 
	and differential forms-valued Schwartz functions (as well as the tempered distributions), i.e., 
  $\slashed{\mathcal{S}}(\R^n)$ and $\mathcal{S}^p(\R^n)$ ($\slashed{\mathcal{S}}^\prime(\R^n)$ and $\mathcal{S}^{\prime p}(\R^n)$),
  and will be denoted by $\mathcal{F}$ as well. 
  
\subsection{Riesz distribution}
  Let $x\in \R^n$. The classical {\it Riesz distribution} \cite{Riesz,GelfandShilov} is defined by 
  \begin{align}\label{eq:ClassicalRiesz}
    r^\lambda(x)\st (x_1^2+\ldots+x_n^2)^{\frac{\lambda}{2}}=\abs{x}^\lambda,
  \end{align}
  where $\lambda\in\C$. It is an analytic function in the complex half-plane 
	$\Re(\lambda)>-n$. Due to the Bernstein-Sato identity 
  \begin{align}\label{eq:BSClassicalRiesz}
    \Delta r^{\lambda+2}(x)=(\lambda+2)(\lambda+n) r^\lambda(x),
  \end{align}
  where $\Delta=\sum\limits_{k=1}^n \partial_k^2$, the meromorphic continuation 
	(with simple poles at $\lambda=-n-2k$ for $k\in\N_0$) 
  of $r^\lambda(x)$ to $\lambda\in\C$ follows. Let us introduce a meromorphic function   
  \begin{align}\label{eq:Constant}
    c_\lambda\st 2^{\lambda+n}\pi^{\frac n2}\Gamma(\frac{\lambda+n}{2})\Gamma(-\frac{\lambda}{2})^{-1},
  \end{align}
  and the standard notation for the Pochhammer symbol  
  \begin{align*}
    (a)_n\st a\cdot (a+1)\cdot\ldots\cdot(a+n-1)
  \end{align*}
  for $n\in\N$ and $a\in\C$. Then a classical result states
  \begin{prop}\label{FourierClassicalRiesz}
    The Fourier transformation of $r^\lambda(x)$ is given by 
    \begin{align*}
      \mathcal{F}(r^\lambda)(\xi)=c_\lambda r^{-\lambda-n}(\xi).
    \end{align*} 
  \end{prop}
  Based on the Bernstein-Sato identity \eqref{eq:BSClassicalRiesz} and a knowledge of the residue 
	for $r^\lambda(x)$ at $\lambda=-n$, see 
  \cite{GelfandShilov}, we get immediately 
  \begin{corollary}
    The residue of $r^\lambda(x)$ at $\lambda=-n-2k$ for $k\in\N_0$ is given by 
    \begin{align*}
      \Res_{\lambda=-n-2k}(r^\lambda(x))= \frac{2\pi^{\frac n2}}{4^k k!\Gamma(\frac n2)(\frac n2)_k} \Delta^k \delta_0(x).
    \end{align*}
  \end{corollary}
  Consequently, the residues of $r^\lambda(x)$ are related to GJMS operators 
	$P_{2N}=\Delta^N$ on $(\R^n,\langle\cdot,\cdot\rangle)$ for any $N\in\N_0$.
 	
%%%%%%%%%%%%%%%%%%%%%%%%%%%%%%%%%%%%%%%%%%% 
\subsection{Riesz distribution for spinors}  
  We proceed with the Riesz distribution for spinors on $(\R^n,\langle\cdot,\cdot\rangle)$, see \cite{CO} .
  We write $\mathbb{S}_n^\pm$ for the irreducible half-spin representations for even $n$ 
  and $\mathbb{S}_n$ for the irreducible spin representation in the case of odd $n$. Then it holds 
	that $\mathbb{S}_n^\pm\simeq \mathbb{S}_{n-1}$ for even $n$, while 
  $\mathbb{S}_n\simeq \mathbb{S}_{n-1}^+\oplus \mathbb{S}_{n-1}^-$ for odd $n$. 
  Let $\Sigma_n$ be the spinor bundle of $(\R^n,\langle\cdot,\cdot\rangle)$ associated to 
  the spin representation $\mathbb{S}_n$ for odd $n$, respectively $\mathbb{S}_n^+$ for even $n$. 
  The Clifford multiplication $\cdot$ is normalized by
  $x\cdot y+y\cdot x=-2\langle x,y\rangle$ for $x,y\in\R^n$. 
  The action of the Dirac operator on spinor fields $\varphi\in\Gamma(\Sigma_n)$ is locally, with respect to the standard basis $\{e_k\}$ of $\R^n$, given by 
  \begin{align*}
    \slashed{D}\varphi=\sum_{k=1}^n e_k\cdot \partial_k\varphi. 
  \end{align*}
  We use the identification of a point $x\in\R^n$ with a vector in $\R^n$, and 
  define the $\End(\Sigma_n)$-valued distribution called the {\it Riesz distribution for spinors}: 
  \begin{align}\label{eq:ClassicalRieszSpinor}
    \slashed{r}^\lambda(x)\st r^{\lambda-1}(x) x\cdot=r^\lambda(x)\frac{x}{\abs{x}}\cdot.
  \end{align}
  In the region $\Re(\lambda)>-n$,
	\eqref{eq:ClassicalRieszSpinor} is an analytic function and satisfies 
  the Bernstein-Sato identity
  \begin{align}\label{eq:BSClassicalRieszSpinor}
      \Delta\slashed{r}^{\lambda+2}(x)=(\lambda+1)(\lambda+n+1)\slashed{r}^{\lambda}(x).
  \end{align}
  This follows from $\slashed{D}\slashed{r}^\lambda(x)=-(\lambda+n-1)r^{\lambda-1}(x)$, 
  $\slashed{D}r^{\lambda-1}(x)=(\lambda-1)\slashed{r}^{\lambda-2}(x)$ and $\slashed{D}^2=-\Delta$. 
  In turn, the equation \eqref{eq:BSClassicalRieszSpinor} implies meromorphic continuation of $\slashed{r}^\lambda(x)$ 
  to the complex plane $\C$ with simple poles at $\lambda=-n-1-2k$ for $k\in\N_0$. 
  
  The Fourier transform preserves the family of Riesz distributions for spinors. 
  \begin{prop}\label{FourierRieszSpinor}
    The Fourier transform of $\slashed{r}^\lambda(x)$ is given by 
    \begin{align}\label{FourierRieszSpinorequality}
      \mathcal{F}(\slashed{r}^\lambda)(\xi)=\slashed{c}_\lambda \slashed{r}^{-\lambda-n}(\xi),
    \end{align}
    where $\slashed{c}_\lambda \st -i\frac{c_{\lambda+1}}{\lambda+1}$.
  \end{prop}
  Since we could not find its proof in the literature, we shall supply it here. 
  \begin{proof}
    Starting on the right side of \eqref{FourierRieszSpinorequality} and using 
		Proposition \ref{FourierClassicalRiesz} together with the fact  
    that $\xi\cdot \mathcal{F}(\varphi)(\xi)=i\mathcal{F}(\slashed{D}\varphi)(\xi)$, we compute 
    \begin{align*}
      \slashed{r}^{-\lambda-n}(\xi)\mathcal{F}(\varphi)(\xi)
        &=r^{-\lambda-n-1}(\xi)\xi\cdot \mathcal{F}(\varphi)(\xi)\\
      &=i(c_{\lambda+1})^{-1}\mathcal{F}(r^{\lambda+1})(\xi)\mathcal{F}(\slashed{D}\varphi)\\
      &=i(c_{\lambda+1})^{-1}\mathcal{F}(\int_{\R^n}r^{\lambda+1}(x-y)\slashed{D}\varphi \dm y) .
    \end{align*}
    We choose a scalar product on $\langle\cdot ,\cdot \rangle_{\Sigma_n}$ on spinors, and also a constant spinor $\phi$.
		Then we have
    \begin{align*}
      \langle\phi,r^{\lambda+1}(x-y)\slashed{D}\varphi\rangle_{\Sigma_n}
        =\langle\phi, (\lambda+1)r^{\lambda-1}(x-y)\sum_{j=1}^n(x_j-y_j)e_j\cdot\varphi\rangle_{\Sigma_n},
    \end{align*}
    and therefore 
    \begin{align*}
      \mathcal{F}(\int_{\R^n}r^{\lambda+1}(x-y)\slashed{D}\varphi \dm y)
        = (\lambda+1)\mathcal{F}(\int_{\R^n}r^{\lambda-1}(x-y)\sum_{j=1}^n(x_j-y_j)e_j\cdot\varphi \dm y).
    \end{align*}
    The proof is complete.
  \end{proof}
  Finally, we recall the residues of $\slashed{r}^\lambda(x)$, see \cite[Proposition $6.3$]{CO}, which correspond  
  to (odd) conformal powers of the Dirac operator $\slashed{D}_{2N+1}=\slashed{D}^{2N+1}$ on $(\R^n,\langle\cdot,\cdot\rangle)$ for any $N\in\N_0$.
  \begin{prop}
    The residue of $\slashed{r}^\lambda(x)$ at $\lambda=-n-2k-1$, for $k\in\N$, is given by 
    \begin{align*}
      \Res_{\lambda=-n-1-2k}(\slashed{r}^{\lambda}(x))= \frac{2\pi^{\frac n2}}{4^k k!\Gamma(\frac n2)(\frac n2)_k} \slashed{D}^{2k+1}\delta_0(x).
    \end{align*}
  \end{prop}

%%%%%%%%%%%%%%%%%%%%%%%%%%
\subsection{Riesz distribution for differential forms}  

  We consider differential forms on $\R^n$. As in the previous section a point $x\in \R^n$ 
	is also regarded as a vector. 
  The inner and exterior products with respect to the vector $x$ are denoted by 
  \begin{align*}
    i_x&\st\sum_{k=1}^n x_k i_{e_k},\quad \varepsilon_x\st \sum_{k=1}^n x_k \varepsilon_{e_k},
  \end{align*}
  respectively. 
  The exterior differential, its co-differential and the form Laplacian act on 
	differential $p$-forms $\Omega^p(\R^n)$ by 
  \begin{align*}
    \dm\st\sum_{k=1}^n \varepsilon_{e_k}\partial_k,\quad \delta\st-\sum_{k=1}^n i_{e_k} \partial_k,\quad 
      \Delta_p\st \dm\delta+\delta\dm=-\Delta,
  \end{align*}
  while similar operators on $\Omega^p(\R^{n-1})$ are denoted by $\dm^\prime,\delta^\prime$ and $\Delta_p^\prime$, respectively.
  
  Now, the Riesz distribution on differential forms \cite{FO} is defined by 
  \begin{align}\label{eq:FormRiesz}
    R_p^\lambda(x)\st r^{\lambda-2}(x)(i_x\varepsilon_x-\varepsilon_x i_x).
  \end{align}
  In the region $\Re(\lambda)>-n$ of $\mathbb{C}$ it is an analytic function 
	and satisfies the following Bernstein-Sato identity
  \begin{align}\label{eq:BSForFormRiesz}
   \Big[(\lambda+2n-2p)(\lambda+2p-2)\delta\dm 
     + (\lambda+2p)(\lambda+2n-2p-2)\dm\delta \Big]
     R_p^{\lambda}(x)&\\
   =-(\lambda-2)(\lambda+n-2)(\lambda+2p)(\lambda+2n-2p)R^{\lambda-2}_p(x)&.
  \end{align}
  This implies the meromorphic continuation of $R_p^{\lambda}(x)$ to $\lambda\in\C$ with simple poles at $\lambda=-n-2k$ for 
  $k\in\N_0$. We shall introduce
  \begin{align}\label{eq:BGCoeff}
    \alpha_\lambda\st \frac n2-p+\lambda,\quad \beta_\lambda \st \frac n2-p-\lambda
		\quad \mbox{for}\quad \lambda\in{\mathbb{C}} ,
  \end{align}
  which are related to Branson-Gover operators $L_{2N}^{(p)}= \alpha_N(\delta\dm)^N+\beta_N(\dm\delta)^N$
  on $(\R^n,\langle\cdot,\cdot\rangle)$ for any $N\in\N$.
  \begin{prop}\label{FourierRieszForm}
    The Fourier transform of $R_p^{\lambda}(x)$ is given by
    \begin{align*}
    \mathcal{F}(R_p^\lambda)(\xi)
       = \bar{c}_\lambda r^{-\lambda-n-2}(\xi)(\alpha_{-\frac{\lambda+n}{2}}i_\xi\varepsilon_\xi
          +\beta_{-\frac{\lambda+n}{2}}\varepsilon_\xi i_\xi),
    \end{align*}
    where $\bar{c}_\lambda\st (\lambda-1)(\lambda-2)c_{\lambda}$.
  \end{prop}
  Finally, we recall that the residues of $R_p^{\lambda}(x)$ correspond 
	to the Branson-Gover operators on $\R^n$.
  \begin{prop}\label{ResiduesForFormRiesz}
    Let $k\in\N_0$. Then the residue of $R_p^{\lambda}(x)$ at $\lambda=-n-2k$ is given by 
    \begin{align*}
      \Res_{\lambda=-n-2k}(R^\lambda_p(x) )
        =\frac{(-1)^k 2\pi^{\frac n2}}{4^{k} k! \Gamma(\frac n2+k+1)}
         [\alpha_k (\delta\dm)^k+\beta_k (\dm\delta)^k]\delta_0(x).
   \end{align*}
  \end{prop}

%%%%%%%%%%%%%%%%%%%%%%%%%%%%%%%%%%%%%%%%%%
%%%
%%%%%%%%%%%%%%%%%%%%%%%%%%%%%%%%%%%%%%%%%%
\section{Bernstein-Sato identity and operator}\label{BSOperatorsAndIdentities}

In the present section we shall prove some Bernstein-Sato identities for distribution kernels associated to conformal 
symmetry breaking operators \cite{KS,MO,K2}. By an abuse of notation, we introduce these 
distribution kernels as adjoints to those appearing 
in the references. The main impact of this choice is that taking Fourier transform 
of these distribution kernels leads to a direct 
contact (without any further dualisation) with a generalized version of singular 
vectors studied in \cite{KOSS,FJS,KKP}. 
%Further note that distribution kernels and their adjoints look almost the same, just in the differential form case there is a crucial difference, cf. Equation \eqref{eq:FormCoDimOneRiesz} and \cite{K2}. 
\subsection{Bernstein-Sato identity and operator in the scalar case}\label{TheScalarCase}
  In this section we prove Bernstein-Sato identity for the distribution kernels associated to 
  conformal symmetry breaking operators acting on functions:  
  \begin{align}\label{eq:ScalarCoDimOneRiesz}
     K^+_{\lambda,\nu}(x^\prime,x_n)&\st \abs{x_n}^{\lambda+\nu-n}(\abs{x^\prime}^2+x_n^2)^{-\nu},\notag\\
     K^-_{\lambda,\nu}(x^\prime,x_n)&\st \sgn(x_n)\abs{x_n}^{\lambda+\nu-n}(\abs{x^\prime}^2+x_n^2)^{-\nu}=x_n K^+_{\lambda-1,\nu}(x^\prime,x_n).
  \end{align}	
  For a detailed analysis of their meromorphic behavior with respect to 
	$(\lambda,\nu)\in\C^2$, see \cite{KS,MO}.   
  
	A method of finding Bernstein-Sato operators, which we follow and which we briefly 
	recall, is based on the discussion in \cite{C1,C}.
  The Knapp-Stein intertwining operator for conformal Lie group, acting on density 
	bundle induced from the character $\gamma$, is given by 
  \begin{align}
     (I_\gamma f)(x)\st (r^{-2\gamma}\star f)(x)=  \int_{\R^n} r^{-2\gamma}(x-y)f(y) \dm y,\label{eq:Knapp-Stein-Scalar}
  \end{align}
  where $f\in\mathcal{S}(\R^n)$.
  It follows from Proposition \ref{FourierClassicalRiesz} that 
  \begin{align*}
    I_{n-\lambda}\circ I_{\lambda}=c_{2\lambda-2n}c_{-2\lambda} \id.
  \end{align*}
  Furthermore, we define the multiplication operator 
  \begin{align}\label{eq:MultOpScalar}
    (M_{x_n}f)(x)\st x_n f(x). 
  \end{align}
  
  \begin{remark}\label{IntertwiningProp}
    We note that both $I_\gamma$ and $M_{x_n}$ are intertwining 
		operators for the conformal Lie groups on $\R^n$ and $\R^{n-1}$, respectively. 
    For more details we refer to \cite{C}. 
  \end{remark}
  Now we define the operator 
  \begin{align}\label{eq:ClercTrickScalar}
    D(\lambda)\st I_{\lambda+1}\circ M_{x_n}\circ I_{n-\lambda}
  \end{align}
  The next statement is remarkable due to the fact that \eqref{eq:ClercTrickScalar}
	is a composition of pseudo-differential operators, cf. Clerc \cite{C}. 
  \begin{prop}\label{ScalarClercTrick}
    The operator $D(\lambda)$ in \eqref{eq:ClercTrickScalar} is a differential operator of order $2$, i.e.,
    \begin{align*}
      D(\lambda)f=-\tilde{c}_\lambda\big[(2\lambda-n)\partial_n f+\Delta(x_n\cdot f)\big]
    \end{align*}
    for any $f\in \mathcal{S}(\R^n)$ and the multiple 
    $\tilde{c}_\lambda\st  c_{-2\lambda-2}c_{2\lambda-2n}$ (cf., \eqref{eq:Constant}.)
  \end{prop}
  We recall its proof.
  \begin{proof}
    In the Fourier image, we compute 
    \begin{align*}
      \mathcal{F}(D(\lambda)f)(\xi)&=\mathcal{F}(I_{\lambda+1}\circ M_{x_n}\circ I_{n-\lambda}f)(\xi)\\
      &=-i c_{-2\lambda-2}c_{2\lambda-2n}r^{2\lambda+2-n}(\xi)\partial_n\big[r^{n-2\lambda}(\xi)\mathcal{F}(f)(\xi) \big].
    \end{align*}
    The identity
    \begin{align*}
     \partial_nr^{n-2\lambda}(\xi)=(n-2\lambda)\xi_n r^{n-2\lambda-2}(\xi)
    \end{align*}
    then implies 
    \begin{align*}
      \mathcal{F}(D(\lambda)f)(\xi)&=-i \tilde{c}_\lambda r^{2\lambda+2-n}(\xi)\big[(n-2\lambda)\xi_nr^{n-2\lambda-2}(\xi)+r^{n-2\lambda}(\xi)\partial_n\big]\mathcal{F}(f)(\xi)\\
      &=\tilde{c}_\lambda\mathcal{F}\big((n-2\lambda)\partial_n f-\Delta(x_n\cdot f) \big)(\xi),
    \end{align*}
    which completes the proof. 
  \end{proof}

  By virtue of Proposition \ref{ScalarClercTrick}, we define the second order differential operator on tempered distributions
	(notice the shift of the parameter $\lambda$) 
  \begin{align*}
    P(\lambda):\mathcal{S}^\prime(\R^n)&\to \mathcal{S}^\prime(\R^n)\\
    f&\mapsto \Delta(x_n\cdot f)+(n-2\lambda)\partial_n f.
  \end{align*}
  Note that by Leibniz's rule we have 
  \begin{align}\label{eq:BSOperator1}
    P(\lambda)=x_n\Delta -(2\lambda-n-2)\partial_n.
  \end{align}
  \begin{remark}\label{IntertwiningProp1}
    The operator $P(\lambda)$, acting on tempered distributions on $\R^n$, is an intertwining differential operator for the conformal Lie 
		group on $\R^{n-1}$, cf. Remark \ref{IntertwiningProp}. The same holds for its iterations used in later sections. 
  \end{remark}
  By the identities in \cite{GelfandShilov}, 
  \begin{align*}
    \partial_n(\abs{x_n}^\lambda)&=\lambda\sgn(x_n)\abs{x_n}^{\lambda-1},\\
     \partial_n(\sgn(x_n)\abs{x_n}^\lambda)&=\lambda\abs{x_n}^{\lambda-1},
  \end{align*}
  a straightforward computation reveals the following result. 
  %I will skip some results in the lemma, not needed here...
  \begin{lem}\label{ScalarLemma}
    The distributions $K^\pm_{\lambda,\nu}(x^\prime,x_n)$ satisfy
    \begin{enumerate}
       \item
       \begin{align*}
               x_nK^\pm_{\lambda,\nu}(x^\prime,x_n)=K^\mp_{\lambda+1,\nu}(x^\prime,x_n),
       \end{align*} 
       \item
       \begin{align*}
          \partial_{n} (K^\pm_{\lambda,\nu}(x^\prime,x_n))
             =(\lambda+\nu-n)K^\mp_{\lambda-1,\nu}(x^\prime,x_n)-2\nu K^\mp_{\lambda,\nu+1}(x^\prime,x_n),
       \end{align*} 
       \item 
       \begin{align*}
         \partial_{i} (K^\pm_{\lambda,\nu}(x^\prime,x_n))=-2\nu x_i K^\pm_{\lambda-1,\nu+1}(x^\prime,x_n)\quad 1\leq i\leq n-1,
       \end{align*}
       \item
       \begin{align*}
          \Delta( K^\pm_{\lambda,\nu}(x^\prime,x_n))
                    = &\, (\lambda+\nu-n-1)_2K^\pm_{\lambda-2,\nu}(x^\prime,x_n)\\
                       - &\, 2\nu(2\lambda-n-2)K^\pm_{\lambda-1,\nu+1}(x^\prime,x_n).
          \end{align*}
    \end{enumerate}
  \end{lem}
  From this Lemma we may conclude 
  \begin{theorem}\label{BSGeneralizedRiesz1}
    The operator $P(\lambda)$ is a spectral shift operator for 
    distribution kernels $K^\pm_{\lambda,\nu}(x^\prime,x_n)$, i.e., 
    \begin{align}\label{eq:BSGeneralizedRiesz1}
      P(\lambda)K^\pm_{\lambda,\nu}(x^\prime,x_n)
        =(\lambda+\nu-n)(\nu-\lambda+1)K^\mp_{\lambda-1,\nu}(x^\prime,x_n),
    \end{align}
    and     
    \begin{align}\label{eq:BSGeneralizedRiesz2}
      P(\frac{\lambda+\nu+1}{2})K^\pm_{\lambda,\nu}(x^\prime,x_n)
        =2\nu(\nu-\lambda+1)K^\mp_{\lambda,\nu+1}(x^\prime,x_n).
    \end{align}
  \end{theorem}
  
  \begin{proof}
    From Lemma \ref{ScalarLemma} we obtain 
    \begin{align*}
      (n-2\lambda^\prime)\partial_n(K^\pm_{\lambda,\nu}(x^\prime,x_n))
        &= (n-2\lambda^\prime)(\lambda+\nu-n)K^\mp_{\lambda-1,\nu}(x^\prime,x_n)
          -2\nu(n-2\lambda^\prime)K^\mp_{\lambda,\nu+1}(x^\prime,x_n),\\
      \Delta(x_n K^\pm_{\lambda,\nu}(x^\prime,x_n))
        &=(\lambda+\nu-n)_2K^\mp_{\lambda-1,\nu}(x^\prime,x_n)-2\nu(2\lambda-n)K^\mp_{\lambda,\nu+1}(x^\prime,x_n),
    \end{align*} 
    hence 
    \begin{align*}
      P(\lambda^\prime)K^\pm_{\lambda,\nu}(x^\prime,x_n)
        &=(\lambda+\nu-n)(-2\lambda^\prime+\lambda+\nu+1)K^\mp_{\lambda-1,\nu}(x^\prime,x_n)\\
      &-2\nu(2\lambda+2\lambda^\prime-2n)K^\mp_{\lambda,\nu+1}(x^\prime,x_n).
    \end{align*}
    Then for $\lambda^\prime\st \lambda$ we conclude 
    \begin{align*}
      P(\lambda)K^\pm_{\lambda,\nu}(x^\prime,x_n)
        =(\lambda+\nu-n)(\nu-\lambda+1)K^\mp_{\lambda-1,\nu}(x^\prime,x_n),
    \end{align*}
    while for $\lambda^\prime\st\frac{\lambda+\nu+1}{2}$ we get 
    \begin{align*}
      P(\frac{\lambda+\nu+1}{2})K^\pm_{\lambda,\nu}(x^\prime,x_n)
        =2\nu(\nu-\lambda+1)K^\mp_{\lambda,\nu+1}(x^\prime,x_n).
    \end{align*}
    The proof is complete.
  \end{proof}

  \begin{remark}
    Assuming the coefficients $A,B$ by 
		$\partial_n$ and $\Delta(x_n\cdot)$ 
		in the formula \eqref{eq:BSOperator1}
		are not known, i.e., $\tilde{P}(\lambda)\st A\partial_n+B \Delta(x_n\cdot)$.
    Then the system of equations 
    \begin{align*}
      (\lambda+\nu-n)A+(\lambda+\nu-n)_2B&=(\lambda+\nu-n)(\nu-\lambda+1),\\
      -2\nu A-2\nu(2\lambda-n)B&=0,
    \end{align*}
     which is equivalent to
    $\tilde{P}(\lambda)K^\pm_{\lambda,\nu}(x^\prime,x_n)=(\lambda+\nu-n)(\nu-\lambda+1)K^\mp_{\lambda-1,\nu}(x^\prime,x_n)$, 
    has a unique solution given by 
    \begin{align*}
      A\st n-2\lambda,\quad B\st 1.
    \end{align*}
    This agrees with Theorem \ref{BSGeneralizedRiesz1}, i.e., $\tilde{P}(\lambda)=P(\lambda)$.
  \end{remark}

  \begin{remark}
    We notice that $K^+_{\lambda,\nu}(x^\prime,x_n)$ generalizes the Riesz distribution $r^\lambda(x)$, 
    \eqref{eq:ClassicalRiesz}, as follows. Once we set $\lambda\st \frac{\mu}{2}+n$ and 
		$\nu\st -\frac{\mu}{2}$, for $\mu\in\C$ such that $\Re(\mu)>-n$,  we get 
   \begin{align}\label{eq:help1}
      K^+_{\frac{\mu}{2}+n,-\frac{\mu}{2}}(x^\prime,x_n)
        =\abs{x_n}^{\frac{\mu}{2}+n-\frac{\mu}{2}-n}(\abs{x^\prime}^2+x_n^2)^{\frac{\mu}{2}}=r^\mu(x).
    \end{align}
    Then Theorem \eqref{eq:BSGeneralizedRiesz2} implies a distributional identity
    \begin{align*}
      P(\frac{\mu}{2}+n)r^\mu(x)=0,
    \end{align*}
    which is equivalent in the light of  
    $P(\frac{\mu}{2}+n)=-(\mu+n-2)\partial_n+x_n\Delta$ and $\partial_n(r^\mu(x))=\mu x_n r^{\mu-2}(x)$ to 
    \begin{align*}
      x_n\big(\Delta( r^\mu(x))-\mu(\mu+n-2)  r^{\mu-2}(x)\big)=0.
    \end{align*}
    Hence we recover the Bernstein-Sato identity \eqref{eq:BSClassicalRiesz} for the Riesz distribution $r^\mu(x)$, i.e., 
    \begin{align*}
      \Delta( r^\mu(x))=\mu(\mu+n-2)r^{\mu-2}(x).
    \end{align*}
  \end{remark}

%%%%%%%%%%%%%%%%%%%%%%%%%%%%%%%%%%%%%%%%%%%%%%%%%%%%%%%%%    
\subsection{Bernstein-Sato identity and operator in the spinor case}\label{TheSpinorCase}
  In the present section we prove a Bernstein-Sato identity for distribution kernels associated to 
	conformal symmetry breaking operators
	acting on spinors:
  \begin{align}\label{eq:SpinorCoDimOneRiesz}
     \slashed{K}^\pm_{\lambda,\nu}(x^\prime,x_n)&\st K^\pm_{\lambda-\frac 12,\nu+\frac 12}(x^\prime,x_n) x\cdot.
  \end{align}
  Similarly to the scalar case, we introduce a Bernstein-Sato operator for $\slashed{K}^\pm_{\lambda,\nu}(x^\prime,x_n)$. 
	First, we recall the Knapp-Stein intertwining operator in the non-compact realization
	of the induced representation of conformal Lie group on spinors \cite{CO}: 
  \begin{align*}
    (\slashed{I}_{\gamma}\varphi)(x)\st (\slashed{r}^{-2\gamma}\star\varphi)(x)  =\int_{\R^n} \slashed{r}^{-2\gamma}(x-y)\varphi(y)\dm y
  \end{align*}
  where $\varphi\in\slashed{\mathcal{S}}(\R^n)$. 
  By Proposition \ref{FourierRieszSpinor}, it follows  
  \begin{align*}
    \slashed{I}_{n-\lambda}\circ \slashed{I}_\lambda=\slashed{c}_{2\lambda-2n}\slashed{c}_{-2\lambda}\id ,
  \end{align*}
  and as in the scalar case we define the operator 
  \begin{align}\label{eq:BSOperatorSpinor}
    \slashed{D}(\lambda)\st \slashed{I}_{\lambda+1}\circ M_{x_n}\circ \slashed{I}_{n-\lambda}
  \end{align}
  with $M_{x_n}$ acting by the scalar multiplication. 
  \begin{theorem}
    The operator $\slashed{D}(\lambda)$ in \eqref{eq:BSOperatorSpinor} is a differential operator of order $2$, i.e., 
    \begin{align}
      \slashed{D}(\lambda)\varphi=\tilde{\slashed{c}}_\lambda\big[(2\lambda-n+1)\partial_{n}\varphi
        +\slashed{D}(e_n\cdot\varphi)+\Delta(x_n\varphi)\big] ,
    \end{align}
    where $\varphi\in\slashed{\mathcal{S}}(\R^n)$ and 
		$\tilde{\slashed{c}}_\lambda\st\slashed{c}_{-2\lambda-2}\slashed{c}_{2\lambda-2n}$. 
  \end{theorem}
  \begin{proof}
    As in the scalar case, we need to understand the right hand side of 
    \begin{align*}
      \mathcal{F}(\slashed{D}(\lambda)\varphi)&=-i\slashed{c}_{-2\lambda-2}\slashed{c}_{2\lambda-2n}\slashed{r}^{2\lambda-n+2}(\xi)
        \partial_{n}\big[ \slashed{r}^{n-2\lambda}(\xi)\mathcal{F}(\varphi)(\xi) \big].
    \end{align*}
    By 
    \begin{align*}
      \partial_{n}(\slashed{r}^{n-2\lambda})(\xi)=(n-2\lambda-1)\xi_n r^{n-2\lambda-3}(\xi)\xi\cdot+r^{n-2\lambda-1}(\xi)e_n\cdot,
    \end{align*}
    it equals to 
    \begin{multline*}
      -i\slashed{c}_{-2\lambda-2}\slashed{c}_{2\lambda-2n}[(n-2\lambda-1)\xi_n r^{-2}(\xi)\xi\cdot \xi \mathcal{F}(\varphi)
        +\xi\cdot\mathcal{F}(e_n\cdot\varphi)+\xi\cdot\xi \cdot\partial_{n}\mathcal{F}(\varphi)]\\
      =-i\slashed{c}_{-2\lambda-2}\slashed{c}_{2\lambda-2n}[-(n-2\lambda-1)\xi_n r^{-2}(\xi) \abs{\xi}^2 \mathcal{F}(\varphi)
        +\xi\cdot\mathcal{F}(e_n\cdot\varphi)-\abs{\xi}^2 \cdot\partial_{n}\mathcal{F}(\varphi)]\\
      =\slashed{c}_{-2\lambda-2}\slashed{c}_{2\lambda-2n}\mathcal{F}\big((2\lambda-n+1)\partial_{n}\varphi+\slashed{D}(e_n\cdot\varphi)+\Delta(x_n\varphi)  \big)
    \end{multline*}
    and the proof is complete.
  \end{proof}
  
  Inspired by the previous Theorem we define, by a shift of the parameter $\lambda$, the operator 
  \begin{align}\label{eq:BSOperatorSpinor1}
    \slashed{P}(\lambda):\slashed{\mathcal{S}}^\prime(\R^n)&\to \slashed{\mathcal{S}}^\prime(\R^n)\notag\\
    \varphi&\mapsto (n-2\lambda+1)\partial_n\varphi+\slashed{D}(e_n\cdot\varphi)+\Delta(x_n \varphi).
  \end{align}
  \begin{remark}
    Using the identity  
    \begin{align*}
      \slashed{D}(e_n\cdot)=-e_n\cdot \slashed{D}^\prime-\partial_n,
    \end{align*}
    we can write 
    \begin{align*}
      \slashed{P}(\lambda)\varphi= P(\lambda)\varphi-e_n\cdot\slashed{D}^\prime\varphi.
    \end{align*}
    Here $\slashed{D}^\prime\st\sum\limits_{k=1}^{n-1}e_k\cdot \partial_k$ is the tangential Dirac operator and $P(\lambda)$ is the scalar Bernstein-Sato operator, see \eqref{eq:BSOperator1}.
  \end{remark}
  \begin{remark}
    Similarly to the scalar case, the operator $\slashed{P}(\lambda)$ is an intertwining 
		differential operator for the conformal Lie group on $\R^{n-1}$, and so is true for its iterations used in later sections. 
  \end{remark}

  Now we collect a few basic properties of the distribution kernels 
	$\slashed{K}^\pm_{\lambda,\nu}(x^\prime,x_n)$ with respect 
	to certain algebraic and differential actions.
  \begin{lem}\label{SpinorLemma} 
	The distribution kernels $\slashed{K}^\pm_{\lambda,\nu}(x^\prime,x_n)$ satisfy the following algebraic and differential identities:
	\begin{enumerate}
	\item
    \begin{align}
      x_n\slashed{K}^\pm_{\lambda,\nu}(x^\prime,x_n)=\slashed{K}^\mp_{\lambda+1,\nu}(x^\prime,x_n),
      %\abs{x}^2\slashed{K}^\pm_{\lambda,\nu}(x^\prime,x_n)&=\slashed{K}^\pm_{\lambda+1,\nu-1}(x^\prime,x_n),\\
     \end{align}
	\item	
	 \begin{align}
	   \partial_n(\slashed{K}^\pm_{\lambda,\nu}(x^\prime,x_n))
             &=(\lambda+\nu-n)\slashed{K}^\mp_{\lambda-1,\nu}(x^\prime,x_n)-2(\nu+\frac 12)\slashed{K}^\mp_{\lambda,\nu+1}(x^\prime,x_n)\nonumber\\
             &+K^\pm_{\lambda-\frac 12,\nu+\frac 12}(x^\prime,x_n)e_n\cdot,
    \end{align}
		\item
			%\partial_k(\slashed{K}^\pm_{\lambda,\nu}(x^\prime,x_n))
      % &=-2(\nu+\frac 12)x_k\slashed{K}^\pm_{\lambda-1,\nu+1}(x^\prime,x_n)
      %    +K^\pm_{\lambda-\frac 12,\nu+\frac 12}(x^\prime,x_n)e_k\cdot,
      %  \quad 1\leq k\leq n-1,\\
    \begin{align} 
	\slashed{D}(\slashed{K}^\pm_{\lambda,\nu}(x^\prime,x_n))
          &=2(\nu+\frac{1-n}{2})K^\pm_{\lambda-\frac 12,\nu+\frac 12}(x^\prime,x_n)\nonumber\\
          &+(\lambda+\nu-n)e_n\cdot \slashed{K}^\mp_{\lambda-1,\nu}(x^\prime,x_n),
	 \end{align} 
	\item
      %\Delta^\prime\slashed{K}^\pm_{\lambda,\nu}(x^\prime,x_n)&=2(\nu+\frac 12)(2\nu-n+4)\slashed{K}^\pm_{\lambda-1,\nu+1}(x^\prime,x_n)
      %  -4(\nu+\frac 12)(\nu+\frac 32)\slashed{K}^\pm_{\lambda,\nu+2}(x^\prime,x_n)\\
      %&-4(\nu+\frac 12)K^\pm_{\lambda-\frac 32,\nu+\frac 32}(x^\prime,x_n)x^\prime\cdot,\\
      \begin{align}
	\Delta(\slashed{K}^\pm_{\lambda,\nu}(x^\prime,x_n))
        &=(\lambda+\nu-n-1)_2\slashed{K}^\pm_{\lambda-2,\nu}(x^\prime,x_n)\nonumber\\
        &-2(\nu+\frac 12)(2\lambda-n-1)\slashed{K}^\pm_{\lambda-1,\nu+1}(x^\prime,x_n)\nonumber\\
        &+2(\lambda+\nu-n)K^\mp_{\lambda-\frac 32,\nu+\frac 12}(x^\prime,x_n)e_n\cdot.
    \end{align}
 \end{enumerate} 
 \end{lem}
  \begin{proof}
    The proof is based on Lemma \ref{ScalarLemma}, and the identities
    \begin{align*}
      \abs{x}^2&=-x\cdot x\cdot,\quad e_k\cdot x\cdot=-x\cdot e_k\cdot-2x_k,\quad \partial_k(x)=e_k,\quad k=1,\ldots n,
    \end{align*}
    with $x=\sum\limits_{k=1}^n x_k e_k$. 
    
    To be more concrete, the first result is obvious by definition 
     of $\slashed{K}^\pm_{\lambda,\nu}(x^\prime,x_n)$. The remaining claim relies 
		on Lemma \ref{ScalarLemma} and the Leibniz-rule. 
     For example, we have 
    \begin{align*}
      \partial_n(\slashed{K}^\pm_{\lambda,\nu}(x^\prime,x_n))
        &=\partial_n(K^\pm_{\lambda-\frac 12,\nu+\frac 12}(x^\prime,x_n))x\cdot
          +K^\pm_{\lambda-\frac 12,\nu+\frac 12}(x^\prime,x_n)\partial_n(x)\cdot\\
      &=(\lambda+\nu-n)\slashed{K}^\mp_{\lambda-1,\nu}(x^\prime,x_n)
        -2(\nu+\frac 12)\slashed{K}^\mp_{\lambda,\nu+1}(x^\prime,x_n)\\
      &+K^\pm_{\lambda-\frac 12,\nu+\frac 12}(x^\prime,x_n) e_n\cdot,
    \end{align*}
    while for $1\leq k\leq n-1$ we get 
    \begin{align*}
      \partial_k(\slashed{K}^\pm_{\lambda,\nu}(x^\prime,x_n))
        &=\partial_k(K^\pm_{\lambda-\frac 12,\nu+\frac 12}(x^\prime,x_n)) x\cdot
          +K^\pm_{\lambda-\frac 12,\nu+\frac 12}(x^\prime,x_n) \partial_k(x)\cdot\\
      &=-2(\nu+\frac 12)x_k K^\pm_{\lambda-\frac 32,\nu+\frac 32}(x^\prime,x_n)x\cdot
        +K^\pm_{\lambda-\frac 12,\nu+\frac 12}(x^\prime,x_n)e_k\cdot\\
      &=-2(\nu+\frac 12)x_k \slashed{K}^\pm_{\lambda-1,\nu+1}(x^\prime,x_n)
        +K^\pm_{\lambda-\frac 12,\nu+\frac 12}(x^\prime,x_n)e_k\cdot.
    \end{align*}
    The remaining assertions then follow easily. 
  \end{proof}
  
  Consequently, the previous Lemma implies the following 
	Bernstein-Sato identity for the distribution kernels $\slashed{K}^\pm_{\lambda,\nu}(x^\prime,x_n)$. 
    
  \begin{theorem}\label{BSGeneralizedRiesz1Spinor}
    The operator $\slashed{P}(\lambda)$ is a spectral shift operator for the 
    distribution kernels $\slashed{K}^\pm_{\lambda,\nu}(x^\prime,x_n)$, i.e.,  
    \begin{align}\label{eq:BSGeneralizedRiesz1Spinor}
      \slashed{P}(\lambda)\slashed{K}^\pm_{\lambda,\nu}(x^\prime,x_n)
        =(\lambda+\nu-n)(\nu-\lambda+1)\slashed{K}^\mp_{\lambda-1,\nu}(x^\prime,x_n).
    \end{align}
  \end{theorem}
  \begin{proof}
    The proof is based on Lemma \ref{SpinorLemma} and the identity
    \begin{align*}
      \slashed{D}(e_n\cdot\varphi)&=-e_n\cdot \slashed{D}(\varphi)-2\partial_n(\varphi).
    \end{align*}
    A straightforward computation shows  
    \begin{align*}
      \slashed{P}(\lambda)\slashed{K}^\pm_{\lambda,\nu}(x^\prime,x_n)
        &=(\lambda+\nu-n)(\nu-\lambda+1)\slashed{K}^\mp_{\lambda-1,\nu}(x^\prime,x_n).
    \end{align*}
    The proof is complete.
  \end{proof}
  
  \begin{remark}
    Regarding the coefficients in equation \eqref{eq:BSOperatorSpinor1}
    by $\partial_n$, $\slashed{D}(e_n\cdot)$ and $\Delta(x_n\cdot)$ 
		as unknown, the ansatz for the operator $\tilde{\slashed{P}}(\lambda)$  
    \begin{align*}
      \tilde{\slashed{P}}(\lambda)\st A\partial_n+B \slashed{D}(e_n\cdot)+C\Delta(x_n\cdot)
    \end{align*}
    leads to the system of equations 
    \begin{align*}
      (\lambda+\nu-n)A-(\lambda+\nu-n)B+(\lambda+\nu-n)_2C&=(\lambda+\nu-n)(\nu-\lambda+1),\\
      -2(\nu+\frac 12) A+4(\nu+\frac 12)B-2(\nu+\frac 12)(2\lambda-n+1)C&=0,\\
      A-2(\nu-\frac{n-3}{2})B+2(\lambda+\nu-n+1)C&=0,
    \end{align*}
    equivalent to 
    $\tilde{\slashed{P}}(\lambda)\slashed{K}^\pm_{\lambda,\nu}(x^\prime,x_n)=(\lambda+\nu-n)(\nu-\lambda+1)\slashed{K}^\mp_{\lambda-1,\nu}(x^\prime,x_n)$. The unique solution of this system 
		is given by 
    \begin{align*}
      A\st n-2\lambda+1,\quad B\st 1,\quad C\st 1,
    \end{align*}
    which agrees with Theorem \ref{BSGeneralizedRiesz1Spinor}, 
		i.e., $\tilde{\slashed{P}}(\lambda)=\slashed{P}(\lambda)$.
  \end{remark}

  \begin{remark}
    The distribution kernels $\slashed{K}^{\pm}_{\lambda,\nu}(x^\prime,x_n)$ generalize 
		the Riesz distribution $\slashed{r}^\lambda(x)$, 
    see \eqref{eq:ClassicalRieszSpinor}, in the sense that for $\mu\in\C$ with $\Re(\mu)>-n$ it holds
    \begin{align}\label{eq:help2}
      \slashed{K}^{+}_{\frac{\mu}{2}+n,-\frac{\mu}{2}}(x^\prime,x_n)=\slashed{r}^\mu(x). 
    \end{align} 
    The last theorem implies that 
    \begin{align*}
      \slashed{P}(\frac{\mu}{2}+n)\slashed{r}^{\mu}(x)=0
    \end{align*}
    as a distributional identity, which is equivalent to
    \begin{align*}
      -(\mu+n-1)\partial_n(\slashed{r}^{\mu}(x))-e_n\slashed{D}(\slashed{r}^{\mu}(x))+x_n\Delta(\slashed{r}^{\mu}(x))=0.
    \end{align*}
    By 
    \begin{align*}
      \partial_n(\slashed{r}^{\mu}(x))&=(\mu-1) \slashed{r}^{\mu-2}(x)+r^{\mu-1}(x) e_n,\\
      \slashed{D}(\slashed{r}^{\mu}(x))&=-(\mu+n-1)r^{\mu-1}(x),
    \end{align*}
    we obtain 
    \begin{align*}
      \Delta(\slashed{r}^\mu(x))=(\mu-1)(\mu+n-1)\slashed{r}^{\mu-2}(x).
    \end{align*}
    This is the Bernstein-Sato identity for $\slashed{r}^\mu(x)$, see \eqref{eq:BSClassicalRieszSpinor}. Independently this follows from   
    $\slashed{D}(r^{\lambda-1}(x))=(\lambda-1)\slashed{r}^{\lambda-2}(x)$, a coupled Bernstein-Sato identity for scalars and spinors, and $\slashed{D}^2=-\Delta$.
  \end{remark}
  
%%%%%%%%%%%%%%%%%%%%%%%%%
\subsection{Bernstein-Sato identity and operator in the form case}
  
  In the present section we prove a Bernstein-Sato identity for distribution kernels associated to conformal symmetry breaking operators on differential 
  forms: 
  \begin{align}\label{eq:FormCoDimOneRiesz}
    K^{(p),\pm}_{\lambda,\nu}(x^\prime,x_n)
       &\st K^\pm_{\lambda-1,\nu+1}(x^\prime,x_n) (i_x\varepsilon_x-\varepsilon_x i_x)i_{e_n}\varepsilon_{e_n},     
  \end{align}
  Similarly to the scalar case, we introduce a Bernstein-Sato operator for $K^{(p),\pm}_{\lambda,\nu}(x^\prime,x_n)$. 
  First, we recall the Knapp-Stein intertwining operator in the non-compact realization 
  of the induced representation of conformal Lie group on differential forms \cite{FO}: 
  \begin{align*}
    (I^p_{\gamma}\omega)(x)\st (R_p^{-2\gamma}\star\omega)(x)  
    =\int_{\R^n} R_p^{-2\gamma}(x-y)\omega(y)\dm y.
  \end{align*}
  By Proposition \ref{FourierRieszForm} we obtain
  \begin{align*}
    I^p_{n-\lambda}\circ I^p_{\lambda}=\bar{c}_{2\lambda-2n}\bar{c}_{-2\lambda}(\lambda-p)(n-p-\lambda)\id.
  \end{align*}
  As in the scalar case we define the operator 
  \begin{align}\label{eq:BSOperatorForm}
    D^p(\lambda)\st I^p_{\lambda+1}\circ M_{x_n}\circ I^p_{n-\lambda},
  \end{align}
  with $M_{x_n}$ acting by the scalar multiplication. 
  \begin{theorem}
    The operator $D^p(\lambda)$ in \eqref{eq:BSOperatorForm} 
    is a differential operator of order $2$, i.e., 
    \begin{align*}
      D^p(\lambda)\omega&=\tilde{\bar{c}}_\lambda\Big[(2\lambda-n)(\lambda-p+1)(\lambda-n+p+1)\partial_n\omega\\
      &+(2\lambda-n)[(\lambda-p+1)\delta (\varepsilon_{e_n}\omega)-(\lambda-n+p+1)\dm (i_{e_n}\omega)]\\
      &+[(\lambda-p+1)(n-\lambda-p)\delta\dm+(\lambda-p)(n-\lambda-p-1)\dm\delta](x_n\cdot\omega)\Big] ,
    \end{align*}
    where $\omega\in\Omega^p(\R^n)$ and 
    $\tilde{\bar{c}}_\lambda\st \bar{c}_{-2\lambda-2}\bar{c}_{2\lambda-2n}$. 
  \end{theorem}
  
    \begin{proof}
    In the Fourier image it follows that 
    \begin{align*}
      \mathcal{F}(D^p(\lambda)\omega)(\xi)&=\mathcal{F}(I^p_{\lambda+1}\circ M_{x_n}\circ I^p_{n-\lambda}\omega)(\xi)\\
      &=-i\bar{c}_{-2\lambda-2}\bar{c}_{2\lambda-2n} r^{2\lambda-n}(\xi) (\alpha_{\lambda+1-\frac n2}i_\xi\varepsilon_\xi+\beta_{\lambda+1-\frac n2}\varepsilon_\xi i_\xi)\times\\
     &\quad\quad\quad\times\partial_n\big[r^{n-2\lambda-2}(\xi)(\alpha_{\frac n2-\lambda}i_\xi\varepsilon_\xi+\beta_{\frac n2-\lambda}\varepsilon_\xi i_\xi)\mathcal{F}(\omega)(\xi) \big].
    \end{align*}
    The identities
    \begin{align*}
     \partial_n(r^{n-2\lambda-2}(\xi))&=(n-2\lambda-2)\xi_n r^{n-2\lambda-4}(\xi),\\
     \partial_n(i_\xi\varepsilon_\xi)&=i_{e_n}\varepsilon_\xi+i_\xi\varepsilon_{e_n}=\xi_n-\varepsilon_\xi i_{e_n}+i_\xi\varepsilon_{e_n},\\
     \partial_n(\varepsilon_\xi i_\xi)&=\varepsilon_{e_n}i_{\xi}+\varepsilon_\xi i_{e_n}=\xi_n-i_\xi\varepsilon_{e_n}+\varepsilon_\xi i_{e_n}
    \end{align*}
    imply
    \begin{align*}
      \mathcal{F}(D^p(\lambda)\omega)(\xi)&=-i\tilde{\bar{c}}_{\lambda}r^{2\lambda-n}(\xi) (\alpha_{\lambda+1-\frac n2}i_\xi\varepsilon_\xi+\beta_{\lambda+1-\frac n2}\varepsilon_\xi i_\xi)\times\\
     &\quad\quad\quad\times\big[(n-2\lambda-2)\xi_nr^{n-2\lambda-4}(\xi)(\alpha_{\frac n2-\lambda}i_\xi\varepsilon_\xi+\beta_{\frac n2-\lambda}\varepsilon_\xi i_\xi)\mathcal{F}(\omega)(\xi)\\
     &\quad\quad\quad\quad+(\alpha_{\frac n2-\lambda}+\beta_{\frac n2-\lambda})\xi_n r^{n-2\lambda-2}(\xi)\mathcal{F}(\omega)(\xi)\\
     &\quad\quad\quad\quad+(\beta_{\frac n2-\lambda}-\alpha_{\frac n2-\lambda})r^{n-2\lambda-2}(\xi)(\varepsilon_\xi i_{e_n}-i_\xi\varepsilon_{e_n})\mathcal{F}(\omega)(\xi)\\
     &\quad\quad\quad\quad+r^{n-2\lambda-2}(\xi) (\alpha_{\frac n2-\lambda}i_{\xi}\varepsilon_\xi +\beta_{\frac n2-\lambda}\varepsilon_{\xi}i_\xi)\partial_n\mathcal{F}(\omega)(\xi) \big].
    \end{align*}
    The substitution
    \begin{align*}
      (\alpha_{\frac n2-\lambda}i_\xi\varepsilon_\xi+\beta_{\frac n2-\lambda}\varepsilon_\xi i_\xi)
        =(\alpha_{\frac n2-\lambda-1}i_\xi\varepsilon_\xi+\beta_{\frac n2-\lambda-1}\varepsilon_\xi i_\xi)
        +(i_\xi\varepsilon_\xi-\varepsilon_\xi i_\xi)
    \end{align*}
    together with 
    \begin{align*}
      \alpha_{\lambda+1-\frac n2}\alpha_{\frac n2-\lambda-1}=\beta_{\lambda+1-\frac n2}\beta_{\frac n2-\lambda-1},\\
      i_\xi\varepsilon_\xi+\varepsilon_\xi i_\xi=\abs{\xi}^2,\quad (i_\xi)^2=0=(\varepsilon_\xi)^2
    \end{align*}
    give
    \begin{align*}
      \mathcal{F}(D^p(\lambda)\omega)(\xi)=-i&\tilde{\bar{c}}_{\lambda}
       \big[(n-2\lambda-2)\alpha_{\lambda+1-\frac n2}\alpha_{\frac n2-\lambda-1}\xi_n\mathcal{F}(\omega)(\xi)\\
     &+(n-2\lambda-2)\xi_n r^{-2}(\xi)(\alpha_{\lambda+1-\frac n2}i_\xi \varepsilon_\xi-\beta_{\lambda+1-\frac n2}\varepsilon_\xi i_\xi)\mathcal{F}(\omega)(\xi)\\
     &+(\alpha_{\frac n2-\lambda}+\beta_{\frac n2-\lambda})\xi_n r^{-2}(\xi)(\alpha_{\lambda+1-\frac n2}i_\xi \varepsilon_\xi+\beta_{\lambda+1-\frac n2}\varepsilon_\xi i_\xi)\mathcal{F}(\omega)(\xi)\\
     &+(\beta_{\frac n2-\lambda}-\alpha_{\frac n2-\lambda})r^{-2}(\xi)(\beta_{\lambda+1-\frac n2}\varepsilon_\xi i_{\xi}\varepsilon_\xi i_{e_n}-\alpha_{\lambda+1-\frac n2}i_{\xi}\varepsilon_\xi i_\xi\varepsilon_{e_n})\mathcal{F}(\omega)(\xi)\\
     &+(\alpha_{\lambda+1-\frac n2}\alpha_{\frac n2-\lambda}i_{\xi}\varepsilon_\xi +\beta_{\lambda+1-\frac n2}\beta_{\frac n2-\lambda}\varepsilon_{\xi}i_\xi)\partial_n\mathcal{F}(\omega)(\xi) \big].
    \end{align*}
    Now, recalling the explicit form of the 
		coefficients $\alpha_\lambda,\beta_\lambda$, cf. \eqref{eq:BGCoeff}, 
    \begin{align*}
      \alpha_{\lambda+1-\frac n2}&=(\lambda-p+1),\quad \alpha_{\frac n2-\lambda}=(n-\lambda-p),
       \quad \alpha_{\frac n2-\lambda-1}=(n-\lambda-p-1),\\
      \beta_{\lambda+1-\frac n2}&=(n-\lambda-p-1),\quad \beta_{\frac n2-\lambda}=(\lambda-p),
       \quad \beta_{\frac n2-\lambda-1}=(\lambda-p+1),
    \end{align*}
    allows to conclude
    \begin{align*}
      (n-2\lambda-2)\alpha_{\lambda+1-\frac n2}&+(\alpha_{\frac n2-\lambda}+\beta_{\frac n2-\lambda})\alpha_{\lambda+1-\frac n2}=2(n-\lambda-p-1)(\lambda-p+1),\\
      -(n-2\lambda-2)\beta_{\lambda+1-\frac n2}&+(\alpha_{\frac n2-\lambda}+\beta_{\frac n2-\lambda})\beta_{\lambda+1-\frac n2}=2(\lambda-p+1)(n-\lambda-p-1),\\
      \varepsilon_\xi i_{\xi}\varepsilon_\xi i_{e_n}&=\abs{\xi}^2\varepsilon_\xi i_{e_n},\quad
        i_{\xi}\varepsilon_\xi i_\xi\varepsilon_{e_n}=\abs{\xi}^2i_\xi\varepsilon_{e_n}.
    \end{align*}
   This finally implies 
    \begin{align*}
      \mathcal{F}(D^p(\lambda)\omega)=\tilde{\bar{c}}_{\lambda}\mathcal{F}\big(&
        (2\lambda-n)(\lambda-p+1)(\lambda-n+p+1)\partial_n\omega\\
      &+(2\lambda-n)[(\lambda-p+1)\delta (\varepsilon_{e_n}\omega)-(\lambda-n+p+1)\dm (i_{e_n}\omega)]\\
      &+[(\lambda-p+1)(n-\lambda-p)\delta\dm+(\lambda-p)(n-\lambda-p-1)\dm\delta](x_n\cdot\omega)  \big),
    \end{align*} 
    which completes the proof. 
  \end{proof}
  
  Let us renormalize the operator $D^p(\lambda)$ and shift the parameter $\lambda$ to $n-\lambda$: 
  \begin{align}\label{eq:BSOperatorForDiffForms}
    P^p(\lambda):\mathcal{S}^{\prime,p}(\R^n)&\to \mathcal{S}^{\prime,p}(\R^n)\\
    \omega&\mapsto
      -(2\lambda-n)(\lambda-n+p-1)(\lambda-p-1)\partial_n\omega\notag\\
      &+(2\lambda-n)[(\lambda-n+p-1)\delta (\varepsilon_{e_n}\omega)-(\lambda-p-1)\dm (i_{e_n}\omega)]\notag\\
      &-\big[(\lambda-n+p-1)(\lambda-p)\delta\dm-(\lambda-n+p)(\lambda-p-1)\dm\delta\big](x_n\cdot\omega)\notag
  \end{align}
  \begin{remark}
    The identities of the form 
    \begin{align*}
      -\Delta&=-\sum_{k=1}^n\partial_k^2=\dm\delta+\delta\dm,
        \quad \delta \varepsilon_{e_n}=-\varepsilon_{e_n}\delta-\partial_n,\quad i_{e_n}\dm=-\dm i_{e_n}+\partial_n,\\
      \dm\delta(x_n\cdot)&=\varepsilon_{e_n}\delta-\dm i_{e_n}+x_n\dm\delta,
        \quad \delta\dm(x_n\cdot)=-\varepsilon_{e_n}\delta+\dm i_{e_n}-2\partial_n+x_n\delta\dm,
    \end{align*}
    allow to write 
    \begin{align}\label{eq:BSOoperatorForm1}
      P^p(\lambda)&= -(2\lambda-n-2)(\lambda-n+p-1)(\lambda-p)\partial_n\notag\\
      &-(2\lambda-n-2)\big[(\lambda-n+p)\varepsilon_{e_n}\delta+(\lambda-p)\dm i_{e_n}  \big]\notag\\
      &-(\lambda-n+p-1)(\lambda-p)x_n \delta \dm-(\lambda-n+p)(\lambda-p-1)x_n\dm \delta.
    \end{align}    
    In terms of $P(\lambda)$, see \eqref{eq:BSOperator1}, it holds
    \begin{align}\label{eq:BSOperator1DiffForm}
      P^p(\lambda)&=(\lambda-n+p-1)(\lambda-p)P(\lambda)-(n-2p)x_n \dm \delta\notag\\
      &-(2\lambda-n-2)\big[(\lambda-n+p)\varepsilon_{e_n}\delta+(\lambda-p)\dm i_{e_n}  \big].
    \end{align}
  \end{remark}
  \begin{remark}
    Similarly to the scalar case, the operator $P^p(\lambda)$ is an intertwining 
    differential operator for the conformal Lie group on $\R^{n-1}$, and so is true for its iterations used in later sections. 
  \end{remark}
  
  Now we present some basic properties of $K^{(p),\pm}_{\lambda,\nu}(x^\prime,x_n)$. First note that 
  \begin{align}\label{eq:CSBDF}
    K^{(p),\pm}_{\lambda,\nu}(x^\prime,x_n)=K^{\pm}_{\lambda,\nu}(x^\prime,x_n)i_{e_n}\varepsilon_{e_n}
      -2K^{\pm}_{\lambda-1,\nu+1}(x^\prime,x_n)\varepsilon_x i_x i_{e_n}\varepsilon_{e_n}.
  \end{align}
  \begin{lem}\label{FormLemma}
    The distribution kernels $K^{(p),\pm}_{\lambda,\nu}(x^\prime,x_n)$ satisfy
    \begin{enumerate}
      \item
      \begin{align*}
        x_n K^{(p),\pm}_{\lambda,\nu}(x^\prime,x_n)=K^{(p),\mp}_{\lambda+1,\nu}(x^\prime,x_n),
      \end{align*}
      \item
      \begin{align*}
        \partial_n(K^{(p),\pm}_{\lambda,\nu}(x^\prime,x_n))&=(\lambda+\nu-n)K^{(p),\mp}_{\lambda-1,\nu}(x^\prime,x_n)
          -2\nu K^{\mp}_{\lambda,\nu+1}(x^\prime,x_n)i_{e_n}\varepsilon_{e_n}\\
        &+4(\nu+1)K^{\mp}_{\lambda-1,\nu+2}(x^\prime,x_n)\varepsilon_x i_x i_{e_n}\varepsilon_{e_n}
          -2K^{\pm}_{\lambda-1,\nu+1}(x^\prime,x_n)\varepsilon_{e_n}i_x i_{e_n}\varepsilon_{e_n},
      \end{align*}
      \item
      \begin{align*}
        \varepsilon_{e_n}\delta(K^{(p),\pm}_{\lambda,\nu}(x^\prime,x_n))
           =2(\lambda-p) K^{\pm}_{\lambda-1,\nu+1}(x^\prime,x_n)\varepsilon_{e_n}i_x i_{e_n}\varepsilon_{e_n},
      \end{align*}
      \item
      \begin{align*}
        \dm(i_{e_n}K^{(p),\pm}_{\lambda,\nu}(x^\prime,x_n))
          &=4(\nu+1)K^{\mp}_{\lambda-1,\nu+2}(x^\prime,x_n)\varepsilon_x i_x i_{e_n}\varepsilon_{e_n}
             -2p K^{\mp}_{\lambda,\nu+1}(x^\prime,x_n)i_{e_n}\varepsilon_{e_n}\\
          &-2(\lambda+\nu-n+1)K^{\pm}_{\lambda-1,\nu+1}(x^\prime,x_n)\varepsilon_{e_n}i_x i_{e_n}\varepsilon_{e_n},
      \end{align*}
      \item
      \begin{align*}
        \dm\delta(K^{(p),\pm}_{\lambda,\nu}(x^\prime,x_n))
          &= -4(\nu+1)(\lambda-p) K^{\pm}_{\lambda-2,\nu+2}(x^\prime,x_n)\varepsilon_x i_x i_{e_n}\varepsilon_{e_n}\\
        &+2(\lambda+\nu-n)(\lambda-p)  K^{\mp}_{\lambda-2,\nu+1}(x^\prime,x_n)\varepsilon_{e_n} i_x i_{e_n}\varepsilon_{e_n}\\
        &+2p(\lambda-p)  K^{\pm}_{\lambda-1,\nu+1}(x^\prime,x_n) i_{e_n}\varepsilon_{e_n},
      \end{align*}
      \item
      \begin{align*}
        \delta\dm(K^{(p),\pm}_{\lambda,\nu}(x^\prime,x_n))
          &=-4(\nu+1)(\lambda-n+p)  K^{\pm}_{\lambda-2,\nu+2}(x^\prime,x_n)\varepsilon_x i_x i_{e_n}\varepsilon_{e_n}\\
        &-(\lambda+\nu-n-1)_2  K^{(p),\pm}_{\lambda-2,\nu}(x^\prime,x_n)\\
        &-2(\lambda+\nu-n)(\lambda-p-2)  K^{\mp}_{\lambda-2,\nu+1}(x^\prime,x_n)\varepsilon_{e_n} i_x i_{e_n}\varepsilon_{e_n}\\
        &+\big[2\nu(2\lambda-n-p-2)-2p(\lambda-\nu-p-2)\big]  K^{\pm}_{\lambda-1,\nu+1}(x^\prime,x_n) i_{e_n}\varepsilon_{e_n}.\\
      \end{align*}
    \end{enumerate}
  \end{lem}
  \begin{proof}
    The first claim follows from definition \eqref{eq:FormCoDimOneRiesz}.     
    As for the remaining properties, we use Equation 
    \eqref{eq:CSBDF} and compute the differential actions on both summands 
		separately. We start with some observations.
    First compute, using Lemma \ref{ScalarLemma} and Leibniz's rule, the identities
    \begin{align*}
      \partial_n (K_{\lambda,\nu}^{\pm}(x^\prime,x_n)H)&=(\lambda+\nu-n)K_{\lambda-1,\nu}^{\mp}(x^\prime,x_n)H
        -2\nu K_{\lambda,\nu+1}^{\mp}(x^\prime,x_n)H\\
      &+K_{\lambda,\nu}^{\pm}(x^\prime,x_n)\partial_n(H)\\
      \partial_k (K_{\lambda,\nu}^{\pm}(x^\prime,x_n)H)&=-2\nu x_k K_{\lambda-1,\nu+1}^{\pm}(x^\prime,x_n)H
        +K_{\lambda,\nu}^{\pm}(x^\prime,x_n)\partial_k(H), 
        \quad 1\leq k\leq n-1,
    \end{align*}
    and conclude   
    \begin{align*}
      \delta(K_{\lambda,\nu}^{\pm}(x^\prime,x_n)H)
        &=2\nu K_{\lambda-1,\nu+1}^{\pm}(x^\prime,x_n)i_x H-(\lambda+\nu-n)K_{\lambda-1,\nu}^{\mp}(x^\prime,x_n)i_{e_n} H\\
      &+K_{\lambda,\nu}^{\pm}(x^\prime,x_n)\delta(H),\\
      \dm(K_{\lambda,\nu}^{\pm}(x^\prime,x_n)H)
        &=-2\nu K_{\lambda-1,\nu+1}^{\pm}(x^\prime,x_n)\varepsilon_x H+(\lambda+\nu-n)K_{\lambda-1,\nu}^{\mp}(x^\prime,x_n)\varepsilon_{e_n} H\\
      &+K_{\lambda,\nu}^{\pm}(x^\prime,x_n)\dm(H),\\
      \delta(\varepsilon_x K_{\lambda,\nu}^{\pm}(x^\prime,x_n)H)
        &=-(n-p+E)K_{\lambda,\nu}^{\pm}(x^\prime,x_n)H-\varepsilon_x\delta(K_{\lambda,\nu}^{\pm}(x^\prime,x_n)H),\\
      \delta(\varepsilon_{e_n} K_{\lambda,\nu}^{\pm}(x^\prime,x_n)H)
        &=-\partial_n(K_{\lambda,\nu}^{\pm}(x^\prime,x_n)H)-\varepsilon_{e_n}\delta(K_{\lambda,\nu}^{\pm}(x^\prime,x_n)H),\\
      \dm(i_x K_{\lambda,\nu}^{\pm}(x^\prime,x_n)H)&=(p+E)K_{\lambda,\nu}^{\pm}(x^\prime,x_n)H-i_x \dm(K_{\lambda,\nu}^{\pm}(x^\prime,x_n)H),\\
      \dm(i_{e_n}K_{\lambda,\nu}^{\pm}(x^\prime,x_n)H)&=\partial_n(K_{\lambda,\nu}^{\pm}(x^\prime,x_n)H)-i_{e_n}\dm(K_{\lambda,\nu}^{\pm}(x^\prime,x_n)H)
    \end{align*}
    for some endomorphism $H$ of differential forms. 
		Here we denote by $E\st\sum\limits_{k=1}^n x_k \partial_k$ the Euler operator and 
    close our observations with 
    \begin{align*}
      E(K_{\lambda,\nu}^{(p),\pm}(x^\prime,x_n))=(\lambda-\nu-n)K_{\lambda,\nu}^{(p),\pm}(x^\prime,x_n).
    \end{align*}
    Now it is straightforward to compute
    \begin{align*}
      \partial_n(K_{\lambda,\nu}^{\pm}(x^\prime,x_n)i_{e_n}\varepsilon_{e_n})&=(\lambda+\nu-n)K_{\lambda-1,\nu}^{\mp}(x^\prime,x_n)i_{e_n}\varepsilon_{e_n}
          -2\nu K_{\lambda,\nu+1}^{\mp}(x^\prime,x_n)i_{e_n}\varepsilon_{e_n},\\
      -2\partial_n(K_{\lambda-1,\nu+1}^{\pm}(x^\prime,x_n)\varepsilon_x i_x i_{e_n}\varepsilon_{e_n})
        &=-2(\lambda+\nu-n)K_{\lambda-2,\nu+1}^{\mp}(x^\prime,x_n)\varepsilon_x i_xi_{e_n}\varepsilon_{e_n}\\
      &+4(\nu+1)K_{\lambda-1,\nu+2}^{\mp}(x^\prime,x_n)\varepsilon_x i_xi_{e_n}\varepsilon_{e_n}\\
      &-2K_{\lambda-1,\nu+1}^{\pm}(x^\prime,x_n)\varepsilon_{e_n} i_x i_{e_n}\varepsilon_{e_n},
    \end{align*}
    \begin{align*}
      \varepsilon_{e_n}\delta(K_{\lambda,\nu}^{\pm}(x^\prime,x_n)i_{e_n}\varepsilon_{e_n})
        &=2\nu K_{\lambda-1,\nu+1}^{\pm}(x^\prime,x_n)\varepsilon_{e_n} i_x i_{e_n}\varepsilon_{e_n},\\
      -2\varepsilon_{e_n}\delta(K_{\lambda-1,\nu+1}^{\pm}(x^\prime,x_n)\varepsilon_x i_x i_{e_n}\varepsilon_{e_n})
        &=-2(\nu-\lambda+p)K_{\lambda-1,\nu+1}^{\pm}(x^\prime,x_n)\varepsilon_{e_n} i_x i_{e_n} \varepsilon_{e_n},
    \end{align*}
    \begin{align*}
      \dm(i_{e_n} K_{\lambda,\nu}^{\pm}(x^\prime,x_n)i_{e_n}\varepsilon_{e_n})&=0,\\
      -2\dm(i_{e_n}K_{\lambda-1,\nu+1}^{\pm}(x^\prime,x_n)\varepsilon_x i_x i_{e_n}\varepsilon_{e_n})
        &=4(\nu+1)K_{\lambda-1,\nu+2}^{\mp}(x^\prime,x_n) \varepsilon_x i_x i_{e_n}\varepsilon_{e_n}\\
      &-2(\lambda+\nu-n+1)K_{\lambda-1,\nu+1}^{\pm}(x^\prime,x_n)\varepsilon_{e_n} i_x i_{e_n} \varepsilon_{e_n}\\
      &-2pK_{\lambda-1,\nu+1}^{\pm}(x^\prime,x_n)i_{e_n}\varepsilon_{e_n},
    \end{align*}
    \begin{align*}
      \dm\delta(K_{\lambda,\nu}^{\pm}(x^\prime,x_n)i_{e_n}\varepsilon_{e_n})
        &=-4(\nu)_2 K_{\lambda-2,\nu+2}^{\pm}(x^\prime,x_n)\varepsilon_x i_x i_{e_n}\varepsilon_{e_n}\\
      &+2\nu(\lambda+\nu-n)K_{\lambda-2,\nu+1}^{\mp}(x^\prime,x_n)\varepsilon_{e_n}i_x i_{e_n}\varepsilon_{e_n}\\
      & +2p\nu K_{\lambda-1,\nu+1}^{\pm}(x^\prime,x_n)i_{e_n}\varepsilon_{e_n},\\
      -2\dm\delta(K_{\lambda-1,\nu+1}^{\pm}(x^\prime,x_n)\varepsilon_x i_x i_{e_n}\varepsilon_{e_n})
        &=4(\nu+1)(\nu-\lambda+p)K_{\lambda-2,\nu+2}^{\pm}(x^\prime,x_n)\varepsilon_x i_x i_{e_n}\varepsilon_{e_n}\\
      &-2(\lambda+\nu-n)(\nu-\lambda+p) K_{\lambda-2,\nu+1}^{\mp}(x^\prime,x_n)\varepsilon_{e_n} i_x i_{e_n}\varepsilon_{e_n}\\
      &-2p(\nu-\lambda+p)K_{\lambda-1,\nu+1}^{\pm}(x^\prime,x_n)i_{e_n}\varepsilon_{e_n},\\
    \end{align*}
    and 
    \begin{align*}
      \delta\dm(K_{\lambda,\nu}^{\pm}(x^\prime,x_n)i_{e_n}\varepsilon_{e_n})
        &=4(\nu)_2 K_{\lambda-2,\nu+2}^{\pm}(x^\prime,x_n)\varepsilon_x i_x i_{e_n}\varepsilon_{e_n}\\
      &-2\nu(\lambda+\nu-n)K_{\lambda-2,\nu+1}^{\mp}(x^\prime,x_n)\varepsilon_{e_n} i_x i_{e_n}\varepsilon_{e_n}\\
      &+2\nu(2\lambda-n-p-2) K_{\lambda-1,\nu+1}^{\pm}(x^\prime,x_n)i_{e_n}\varepsilon_{e_n}\\
      &-(\lambda+\nu-n-1)_2K_{\lambda-2,\nu}^{\pm}(x^\prime,x_n)i_{e_n}\varepsilon_{e_n},\\
      -2\delta\dm(K_{\lambda-1,\nu+1}^{\pm}(x^\prime,x_n)\varepsilon_x i_x i_{e_n}\varepsilon_{e_n})
        &=-4(\nu+1)(\lambda+\nu-n+p)K_{\lambda-2,\nu+2}^{\pm}(x^\prime,x_n)\varepsilon_x i_x i_{e_n}\varepsilon_{e_n}\\
      &+2(\lambda+\nu-n-1)_2 K_{\lambda-3,\nu+1}^{\pm}(x^\prime,x_n)\varepsilon_x i_x i_{e_n}\varepsilon_{e_n}\\
      &-2(\lambda+\nu-n)(\lambda-\nu-p-2)K_{\lambda-2,\nu+1}^{\mp}(x^\prime,x_n)\varepsilon_{e_n} i_x i_{e_n}\varepsilon_{e_n}\\
      &-2p(\lambda-\nu-p-2)K_{\lambda-1,\nu+1}^{\pm}(x^\prime,x_n) i_{e_n}\varepsilon_{e_n}.
    \end{align*}
    This completes the proof.
  \end{proof}

  \begin{theorem}\label{BSIdentityForDiffForms}
    The distribution kernels $K_{\lambda,\nu}^{(p),\pm}(x^\prime,x_n)$ satisfy
    \begin{align*}
      P^{(p)}(\lambda)K_{\lambda,\nu}^{(p),\pm}(x^\prime,x_n)=(\lambda+\nu-n)(\nu-\lambda+1)(\lambda-n+p-1)(\lambda-p)K_{\lambda-1,\nu}^{(p),\mp}(x^\prime,x_n).
    \end{align*}
  \end{theorem}
  \begin{proof}
    We use Equation \eqref{eq:BSOoperatorForm1} for the operator $P^p(\lambda)$ and Lemma \ref{FormLemma}. 
    The statement is based on collecting terms contributing 
		to $K^{(p),\mp}_{\lambda-1,\nu}(x^\prime,x_n)$, 
    $K^{\mp}_{\lambda,\nu+1}(x^\prime,x_n)i_{e_n}\varepsilon_{e_n}$, $K^{\mp}_{\lambda-1,\nu+2}(x^\prime,x_n)\varepsilon_x i_x i_{e_n}\varepsilon_{e_n}$ 
    and $K^{\pm}_{\lambda-1,\nu+1}(x^\prime,x_n)\varepsilon_{e_n}i_x i_{e_n}\varepsilon_{e_n}$, respectively. This gives 
    \begin{multline*}
       -(2\lambda-n-2)(\lambda-n+p-1)(\lambda-p)(\lambda+\nu-n)
         +(\lambda-n+p-1)(\lambda-p)(\lambda+\nu-n-1)_2\\
       =(\lambda+\nu-n)(\nu-\lambda+1)(\lambda-n+p-1)(\lambda-p),
    \end{multline*}
    \begin{multline*}
      2\nu(2\lambda-n-2)(\lambda-n+p-1)(\lambda-p)+2p(2\lambda-n-2)(\lambda-p)-2p(\lambda-n+p)(\lambda-p-1)_2\\
      -(\lambda-n+p-1)(\lambda-p)\big[2\nu(2\lambda-n-p-2)-2p(\lambda-\nu-p-2)  \big]=0,
    \end{multline*}
    \begin{multline*}
      -4(\nu+1)(2\lambda-n-2)(\lambda-n+p-1)(\lambda-p)-4(\nu+1)(2\lambda-n-2)(\lambda-p)\\
      +4(\nu+1)(\lambda-n+p)(\lambda-p-1)_2+4(\nu+1)(\lambda-n+p-1)_2(\lambda-p)=0,
    \end{multline*}
    \begin{multline*}
      2(2\lambda-n-2)(\lambda-n+p-1)(\lambda-p)-2(2\lambda-n-2)(\lambda-n+p)(\lambda-p)\\
      +2(2\lambda-n-2)(\lambda-p)(\lambda+\nu-n+1)-2(\lambda-n+p)(\lambda-p-1)_2(\lambda+\nu-n)\\
      +2(\lambda-n+p-1)(\lambda-p)(\lambda+\nu-n)(\lambda-p-2)=0,
    \end{multline*}
    and the proof is complete.
  \end{proof}
  
  \begin{remark}\label{NonUniqueFormBS}
    Let us define the operator 
    \begin{align*}
      \tilde{P}^p(\lambda)\st A\partial_n+B \varepsilon_{e_n}\delta+C\dm i_{e_n}+D x_n \dm\delta+ Ex_n \delta\dm
    \end{align*}
    for some unknown $A,B,C,D,E$. The equation 
    \begin{align*}
      \tilde{P}^p(\lambda)K^{\pm,(p)}_{\lambda,\nu}(x^\prime,x_n)=(\lambda+\nu-n)(\nu-\lambda+1)(\lambda-n+p-1)(\lambda-p)K^{\mp,(p)}_{\lambda+1,\nu}(x^\prime,x_n)
    \end{align*}
    is equivalent to the following system for $A,B,C,D,E$ and 
		$c\st (\lambda+\nu-n)(\nu-\lambda+1)(\lambda-n+p-1)(\lambda-p)$: 
    \begin{align*}
      (\lambda+\nu-n) A-(\lambda+\nu-n-1)_2 E&=c,\\
      -2\nu A+2\nu p D+2\nu(2\lambda-n-p-2)E-2p C&\\
      -2(\nu-\lambda+p)p D-2(\lambda-\nu-p-2)p E&=0,\\
      2\nu B +2\nu (\lambda+\nu-n)D-2\nu(\lambda+\nu-n)E-2A-2(\nu-\lambda+p) B&\\
      -2(\lambda+\nu-n+1)C-2(\lambda+\nu-n)(\nu-\lambda+p)D&\\
      -2(\lambda+\nu-n)(\lambda-\nu-p-2)E&=0,\\
      -4(\nu)_2 D+4(\nu)_2 E+4(\nu+1)A+4(\nu+1)C+4(\nu+1)(\nu-\lambda+p)D&\\
      -4(\nu+1)(\lambda+\nu-n+p)E&=0,\\
      -2(\lambda+\nu-n)A+2(\lambda+\nu-n-1)_2 E&=-2c.
    \end{align*}
    Here the contributions are sorted again according to 
		$K^\mp_{\lambda-1,\nu}(x^\prime,x_n)i_{e_n}\varepsilon_{e_n}$, $K^\mp_{\lambda,\nu+1}(x^\prime,x_n)i_{e_n}\varepsilon_{e_n}$, 
    $K^\pm_{\lambda-1,\nu+1}(x^\prime,x_n)\varepsilon_{e_n} i_x i_{e_n}\varepsilon_{e_n}$, 
    $K^\mp_{\lambda-1,\nu+2}(x^\prime,x_n)\varepsilon_x i_x i_{e_n}\varepsilon_{e_n}$ 
    and $K^\mp_{\lambda-2,\nu+1}(x^\prime,x_n)\varepsilon_x i_x i_{e_n}\varepsilon_{e_n}$. 
    Note that in this case the system does not have a unique solution, we could choose either 
		$B,C$ or $D$ as a free parameter. Choosing one of them as in 
    Equation \eqref{eq:BSOoperatorForm1} will determine the other two and we will recover $P^p(\lambda)$. 
    This observation is a reflection of the fact that in general Bernstein-Sato operators are 
    not unique in contrast with the uniqueness of the Bernstein-Sato polynomial. 
  \end{remark}
  
%%%%%%%%%%%%%%%%%%%%%%%%%%%%%%%%%%%%%%%
%%%
%%%%%%%%%%%%%%%%%%%%%%%%%%%%%%%%%%%%%%%
\section{Applications of Bernstein-Sato identities and operators}\label{ComAndApps}

  In the final section we highlight different origins of Bernstein-Sato operators. Furthermore, we 
  discuss several applications related to conformal symmetry breaking differential operators \cite{KOSS, FJS, KKP}. 
  As a consequence we shall observe that the Bernstein-Sato operators 
  recover conformal symmetry breaking differential operators for functions, 
	spinors and differential forms by partially new formulas. They differ from 
  the known formulas in their product structure expansion, which gives  
	a nice symmetric way of organizing their rather complicated structure.

%%%%%%%%%%%%%%%%%%%%%%%%%%%%%%%%%%%%%%%  
  \subsection{Origins of the Bernstein-Sato operator - the scalar case}
    We discuss some other origins of the Bernstein-Sato operator $P(\lambda)$, see \eqref{eq:BSOperator1}, for functions 
    (less is known for $\slashed{P}(\lambda)$ and $P^p(\lambda)$).
    To   our best knowledge, there are three other approaches to construct the Bernstein-Sato operator 
		\cite{MOZ,GZ,GW}.
    
    {\bf Representation theory:} The differential action on the
    induced representation $\pi_\nu$ of the Casimir element $C$ for 
		conformal Lie algebra $\mathfrak{o}(1,n+1)$ is given in \cite[Equation $5.1$]{MOZ}:
    \begin{align*}
      C(\mu)\st \dm\pi_\mu(C)=x_n^2\Delta+2(\mu+1)x_n\partial_n+(\mu+\frac n2)(\mu-\frac n2+1).
    \end{align*}
    The relationship to $P(\lambda)$ is 
    \begin{align}
      x_nP(\lambda)=C(-\lambda+\frac n2)-(\lambda-n)(\lambda-1).
    \end{align}
    Note that by Theorem \ref{BSGeneralizedRiesz1} we obtain 
    \begin{align*}
       C(-\lambda+\frac n2) K^\pm_{\lambda,\nu}(x^\prime,x_n)=-\nu(n-1-\nu)K^\pm_{\lambda,\nu}(x^\prime,x_n),
    \end{align*}
    hence $K^\pm_{\lambda,\nu}(x^\prime,x_n)$ is an eigendistribution of $C(-\lambda+\frac n2)$. 
    
    {\bf Hyperbolic metric:} The eigen-equation associated to the Laplace operator 
		for the hyperbolic metric $g_{hyp}=x_n^{-2}(\dm x_1^2+\cdots+\dm x_n^2)$ on
		the hyperbolic space also 
    induces the operator $P(\lambda)$, see \cite{GZ}. In particular, 
    \begin{align}
      x_n^{n-s}P(s-1)=(\Delta_{g_{hyp}}-s(n-1-s))x_n^{n-1-s}.
    \end{align}
    Here we used $\Delta_{g_{hyp}}=x_n^2\Delta-(n-2)x_n\partial_n$, where 
		$\Delta=\sum\limits_{k=1}^n\partial_k^2$ is the Laplace operator associated to 
    the metric $x^2_ng_{hyp}$.  
    
    {\bf Tractor calculus:} The invariant pairing of the tractor-D operator 
		$D_A$ with the scale tractor $I_A$ in the Poincar\'e-metric gives, see 
		\cite[Section $5.6$]{GW}, 
    \begin{align}
      P(s)=-I\cdot D.
    \end{align}
    Here the parameter $s$ on the right hand side of the previous equation corresponds 
		to the weight of weighted tractor bundle. 
    Note that the {\it degenerate Laplacian} $I\cdot D$ is of more general 
		nature then our $P(\lambda)$ requires.

  \subsection{Origins of the Bernstein-Sato operator - the spinor case}
    In comparison to the scalar case much less is known for the operator $\slashed{P}(\lambda)$, see \eqref{eq:BSOperatorSpinor1}, 
    in the available literature. 
      
    {\bf Hyperbolic metric:} The eigen-equation associated to the square of the Dirac operator \cite{GMP1}
		for the hyperbolic metric $g_{hyp}=x_n^{-2}(\dm x_1^2+\cdots+\dm x_n^2)$ also 
    induces the operator $\slashed{P}(\lambda)$. In particular, the eigen-equation 
		$\slashed{D}^{g_{hyp}}-i\lambda=0$ on $\Gamma(\Sigma^{g_{hyp}}_n)$, where 
    $\Sigma^{g_{hyp}}_n$ denotes the spinor bundle associated to the hyperbolic space, is equivalent, 
    via conformal covariance ($\bar{g}=\dm x_1^2+\cdots+\dm x_n^2$) of the Dirac operator, to
    $D(\bar{g})-i\lambda=0$ on $\Gamma(\Sigma_n)$. We have
    \begin{align*}
      D(\bar{g})\st x_n\slashed{D}^\prime+x_n D_{\mathcal{N}}-\frac{n-1}{2}e_n\cdot,
    \end{align*}
    where $\slashed{D}^\prime= \sum\limits_{k=1}^{n-1}e_k\cdot\partial_k$ and 
		$D_{\mathcal{N}}\st e_n\cdot\partial_n$. Now, we define the 
    operator $D(\mu)$ by the operator equation 
		$[D(\bar{g})^2+(\mu-\frac{n-1}{2})]x_n^\mu=-x_n^\mu D(\mu)$. It then follows  
    \begin{align}
      D(\mu)=x_n^2\Delta-2(\mu-\frac{n-1}{2})x_n\partial_n-x_n e_n\cdot \slashed{D}=x_n\slashed{P}(\mu+\frac 32). 
    \end{align}

  \subsection{Origins of the Bernstein-Sato operator - the form case}
    In comparison to the scalar and spinor cases even less is known about 
		the operator $P^p(\lambda)$, see \eqref{eq:BSOperatorForDiffForms}, 
    in the available literature. A potential origin of $P^p(\lambda)$ 
		in the construction of the hyperbolic metric \cite{AG} 
		seems not to be well established; to our best knowledge this 
		approach leads to an operator different from $P^p(\lambda)$. 
		It is not clear to the authors if this discrepancy can be explained
		by the non-uniqueness phenomenon mentioned in Remark \ref{NonUniqueFormBS}.

%%%%%%%%%%%%%%%%%%%%%%%%%%%%%%%%%%%%%%%%%%%%%%%%%%%%%%%%%%%%%%    
  \subsection{Conformal symmetry breaking differential operators for functions}
  The operator $P(\lambda)$, see \eqref{eq:BSOperator1}, recovers 
	conformal symmetry breaking differential operators \cite{J1,KS, KOSS} 
  \begin{align}\label{eq:ResidueFamily}
    D_N(\lambda):\mathcal{C}^\infty(\R^n)\to\mathcal{C}^\infty(\R^{n-1}),
  \end{align}
  which are given by 
  \begin{align}
    D_{2N}(\lambda)&=\sum_{k=0}^N a_k^{(N)}(-\lambda) (\Delta^\prime)^k \iota^* \partial_n^{2N-2k},\notag\\
    D_{2N+1}(\lambda)&=\sum_{k=0}^N b_k^{(N)}(-\lambda) (\Delta^\prime)^k \iota^* \partial_n^{2N-2k+1}\label{eq:SBDOScalar}
  \end{align}
  with 
  \begin{align}
    a_j^{(N)}(\lambda)&\st \frac{(-2)^{N-j}N!}{j!(2N-2j)!}\prod_{k=j}^{N-1}(2\lambda-4N+2k+n+1),\notag\\
    b_j^{(N)}(\lambda)&\st \frac{(-2)^{N-j}N!}{j!(2N-2j+1)!}\prod_{k=j}^{N-1}(2\lambda-4N+2k+n-1). \label{eq:GegenbauerCoeff}
  \end{align}
  Here $\Delta^\prime\st \sum\limits_{k=1}^{n-1}\partial_k^2$ denotes the tangential Laplacian 
	for the embedding $\iota:\R^{n-1}\to \R^n$, $\R^{n-1}\ni x^\prime\mapsto (x^\prime,0)\in\R^n$. 
  As mentioned in the introduction, the families (conformal symmetry breaking differential operators) $D_{2N}(\lambda)$ interpolate between 
  GJMS-operators on $\R^{n-1}$ and $\R^n$ 
  equipped with the flat Euclidean metric, respectively. More precisely it holds
  \begin{align}\label{eq:ScalarFamiliesVsGJMS}
    D_{2N}(-\frac{n-1}{2}+N)&=(\Delta^\prime)^N\iota^*,\notag\\
    D_{2N}(-\frac{n}{2}+N)&=\iota^* (\Delta)^N.
  \end{align}
  This appearance of GJMS-operators is part of a sequence of factorization identities for $D_N(\lambda)$, see \cite{J1}.

  Now let us define a family ($N\in\N_0$, $\lambda\in\C$) of differential operators on functions (termed {\it Bernstein-Sato family for functions})
  \begin{align}\label{eq:BSFamily}
    D^{BS}_N(\lambda)\st \iota^*P(\lambda-N+1)\circ \cdots\circ P(\lambda). 
  \end{align}
  This definition, although used in a different setting, already appeared in \cite{C}.

  \begin{example}\label{LowOrderScalarBSFamilyScalar}
    We present the first- and second-order relations between 
		$D^{BS}_N(\lambda)$ and $D_N(\lambda)$. First of all, 
		recall from \eqref{eq:SBDOScalar} that 
    \begin{align*}
      D_1(\lambda)&=\iota^*\partial_n,\\
      D_2(\lambda)&=\Delta^\prime\iota^*+(2\lambda-n+3)\iota^*\partial_n^2.
    \end{align*}
    A direct computation shows
    \begin{align*}
      D^{BS}_1(\lambda)&=\iota^* P(\lambda)=(n-2\lambda+2)\iota^*\partial_n\\
      &=(n-2\lambda+2) D_1(n-\lambda),\\
      D^{BS}_2(\lambda)&=\iota^* P(\lambda-1)\circ P(\lambda)\\
      &=(n-2\lambda+4)\iota^*\partial_n [(n-2\lambda+2)\partial_n+x_n\Delta]\\
      &=(n-2\lambda+4)[(n-2\lambda+3)\iota^*\partial^2_n+\Delta^\prime\iota^*]\\
      &=(n-2\lambda+4) D_2(n-\lambda).
    \end{align*}
  \end{example}
  
  Now we shall discuss the relationship between differential operators 
	$D^{BS}_N(\lambda)$ and $D_N(\lambda)$, see also
  \cite{C} where the proof follows 
	from the property of uniqueness of intertwining differential operators. 
	We present a direct proof of this fact. 
  \begin{theorem}\label{BSFamily}
    Let $N\in\N$. Then we have 
    \begin{align*}
      D^{BS}_{2N}(\lambda)&=(-2)^{N}(\lambda-\frac n2-2N)_N (2N-1)!!D_{2N}(n-\lambda),\\
      D^{BS}_{2N+1}(\lambda)&= (-2)^{N+1}(\lambda-\frac n2-2N-1)_{N+1} (2N+1)!!D_{2N+1}(n-\lambda).
    \end{align*}
  \end{theorem}
  \begin{proof}
    The proof goes by induction. Recall Example \ref{LowOrderScalarBSFamilyScalar} for 
		the lowest order relationship.
    By induction we compute 
    \begin{align*}
      D^{BS}_{2N}(\lambda)&=D^{BS}_{2N-1}(\lambda-1)\circ P(\lambda)\\
      &=(-2)^{N}(\lambda-\frac n2-2N)_{N} (2N-1)!!D_{2N-1}(n-\lambda+1)\circ P(\lambda).
    \end{align*}
    Since 
    \begin{align*}
      \iota^* \partial_n^{2N-2k-1}[(2\lambda-n+2)\partial_n+x_n\Delta]&=(2\lambda-n+2)\iota^*\partial_n^{2N-2k}\\
      &+(2N-2k-1)[\Delta^\prime\iota^*\partial_n^{2N-2k-2}+\iota^*\partial_n^{2N-2k}],
    \end{align*}
    we may conclude
    \begin{align*}
      D^{BS}_{2N}(\lambda)&=(-2)^{N}(\lambda-\frac n2-2N)_{N} (2N-1)!!\Big[A_0(\lambda)\iota^*\partial_n^{2N}+A_N(\lambda)(\Delta^\prime)^N\\
      &+\sum_{k=1}^{N-1}A_k(\lambda)(\Delta^\prime)^k\iota^*\partial_n^{2N-2k}  \Big]
    \end{align*}
    with
    \begin{align*}
      A_0(\lambda)\st&(n-2\lambda+2N+1)b_0^{(N-1)}(\lambda-n-1),\\
      A_N(\lambda)\st& b_{N-1}^{(N-1)}(\lambda-n-1),\\
      A_k(\lambda)\st&(n-2\lambda+2N-2k+1)b_k^{(N-1)}(\lambda-n-1)+(2N-2k+1)b_{k-1}^{(N-1)}(\lambda-n-1).
    \end{align*}
    It follows from \eqref{eq:GegenbauerCoeff} that for $0\leq k\leq N$ holds
    \begin{align*}
        A_k(\lambda)\st&a_k^{(N)}(\lambda-n),
    \end{align*}
    hence we get 
    \begin{align*}
      D^{BS}_{2N}(\lambda)&=(-2)^{N}(\lambda-\frac n2-2N)_{N} (2N-1)!!D_{2N}(n-\lambda).
    \end{align*}
    The remaining statement is proved analogously. The proof is complete. 
  \end{proof}
  
  An immediate consequence of the last result is 
  \begin{corollary}\label{RecurrenceForJuhl}
  Assuming $N\in\N_0$, we have
   \begin{align*}
     D_{2N-1}(n-\lambda+1)\circ P(\lambda)&=D_{2N}(n-\lambda),\\
     D_{2N}(n-\lambda+1)\circ P(\lambda)&=-(2N+1)(2\lambda-n-2N-2) D_{2N+1}(n-\lambda).
   \end{align*}
  \end{corollary} 
  
  \begin{remark}\label{RemarkAboutDifferentProofs}
    To our best knowledge there are two other ways to compute $D^{BS}_{N}(\lambda)$. 
		The first one is based on Fourier transform, cf. proof of Proposition 
		\ref{ScalarClercTrick} for $N=1$. 
    To this aim one needs the following identities:
    \begin{align*}
      \partial_n^{2N}(r^{n-2\lambda}(\xi))&=(2N-1)!!2^{N}(\frac n2-\lambda-N+1)_{N}r^{n-2\lambda-4N}(\xi)\times\\
      &\times\sum_{k=0}^N  a_{N-k}^{(N)}(\lambda-n+2N) r^{2N-2k}(\xi^\prime)\xi_n^{2k},\\
      \partial_n^{2N+1}(r^{n-2\lambda}(\xi))&=(2N-1)!!2^{N+1}(\frac n2-\lambda-N)_{N+1}\xi_n r^{n-2\lambda-4N-2}(\xi)\times\\
      &\times\sum_{k=0}^N b_{N-k}^{(N)}(\lambda-n+2N+1) r^{2N-2k}(\xi^\prime)\xi_n^{2k}.
    \end{align*}
    We note the appearance of coefficients \eqref{eq:GegenbauerCoeff}. 
    
    Another way is based on commutators 
    \begin{align*}
      [P(\lambda),\partial_n]=-\Delta,\quad [P(\lambda),\Delta]=-2\partial_n\Delta
    \end{align*}
    and the fact that $\iota^* P(\lambda)=(n-2\lambda+2)\iota^*\partial_n$. 
    
    Both approaches computing $D_{2N}^{BS}(\lambda)$ are computationally rather 
		tedious and we will not give any more detail here.
  \end{remark}
  
%%%%%%%%%%%%%%%%%%%%%%%%%%%%%%%%%%%%%%%%%%%%%%%%%%%%%%%%%%%

  \subsection{Conformal symmetry breaking differential operators for spinors}
    Here we discuss how the operator $\slashed{P}(\lambda)$, see 
		\eqref{eq:BSOperatorSpinor1}, by its iterations recovers conformal symmetry 
		breaking differential operators for spinors
    \begin{align*}
      \slashed{D}_N(\lambda): \mathcal{C}^\infty(\R^n,\Sigma_n)\to \begin{cases}\mathcal{C}^\infty(\R^{n-1},\Sigma_{n-1})&, n\text{ even },\\
          \mathcal{C}^\infty(\R^{n-1},\Sigma^+_{n-1}\oplus \Sigma^-_{n-1})&, n\text{ odd}\end{cases}
    \end{align*}
    introduced in \cite{KOSS} (note the wrong sign of $\lambda$ in the pre-factor in the reference), and later appearing
		in \cite{MO}. They are given by 
    \begin{align*}
      \slashed{D}_{2N}(\lambda)&\st D_{2N}(\lambda+\frac 12)+2N D_{2N-1}(\lambda+\frac 12)\slashed{D}^\prime (e_n\cdot),\\
      \slashed{D}_{2N+1}(\lambda)&\st (2\lambda-n+2N+2)D_{2N+1}(\lambda+\frac 12)(e_n\cdot)+D_{2N}(\lambda+\frac 12)\slashed{D}^\prime ,\\
    \end{align*}
    where $D_N(\lambda)$ (note that we will mean by $\iota^*$ just restriction) are the conformal symmetry breaking differential operators 
		\eqref{eq:SBDOScalar} and 
    $\slashed{D}^\prime=\sum\limits_{k=1}^{n-1}e_k\cdot \partial_k$ is the 
    tangential Dirac operator. The family $\slashed{D}_{2N+1}(\lambda)$ interpolates between conformal powers of the 
    Dirac operator on $\R^{n-1}$ and $\R^n$ equipped with the flat Euclidean metric, respectively. More precisely, it holds
    \begin{align*}
      \slashed{D}_{2N+1}(\frac{n-1}{2}-\frac 12-N)&= (\slashed{D}^\prime)^{2N+1}\iota^* ,\\
      \slashed{D}_{2N+1}(\frac n2-\frac 12-N)&=(-1)^{N}\iota^* \slashed{D}^{2N+1}
    \end{align*}
    This appearence of conformal powers of the Dirac operators is part of a sequence of factorization identities for $\slashed{D}_{N}(\lambda)$, see \cite{FS}.
    
    Now we define the family for $N\in\N_0$ and $\lambda\in\C$ of differential 
		operators on spinors (termed {\it Bernstein-Sato family for spinors}) by the composition
    \begin{align*}
      \slashed{D}_N^{BS}(\lambda)\st \iota^* \slashed{P}(\lambda-N+1)\circ\cdots\circ \slashed{P}(\lambda)
    \end{align*}
    This definition goes again in the spirit of \cite{C}. 
%    We state some low-order relations between $\slashed{D}_N^{BS}(\lambda)$ and $\slashed{D}_N(\lambda)$.
    \begin{example}\label{LowOrderScalarBSFamilySpinor}
      We present the relationship for the first and the second-order families 
			$\slashed{D}_N^{BS}(\lambda)$ and $\slashed{D}_N(\lambda)$. 
      Firstly, via \eqref{eq:SBDOScalar} we recall  
      \begin{align*}
        \slashed{D}_{1}(\lambda)&=(2\lambda-n+2)\partial_n (e_n\cdot)+\slashed{D}^\prime,\\
        \slashed{D}_2(\lambda)&=D_2(\lambda+\frac 12)+2D_1(\lambda+\frac 12)\slashed{D}^\prime(e_n\cdot)\\
        &=\Delta^\prime+(2\lambda-n+4)\iota^*\partial_n^2+2\iota^*\partial_n \slashed{D}^\prime(e_n\cdot)
      \end{align*}
      The Bernstein-Sato families of first and second-order, respectively, read as
      \begin{align*}
        \slashed{D}_1^{BS}(\lambda)&=\iota^* \slashed{P}(\lambda)=\iota^*[(n-2\lambda+2)\partial_n+x_n\Delta-e_n\slashed{D}^\prime]\\&=-e_n\cdot \slashed{D}_{1}(n-\lambda),\\
        \slashed{D}_2^{BS}(\lambda)&\st \iota^*\slashed{P}(\lambda-1)\circ \slashed{P}(\lambda)\\
        &=[(n-2\lambda+4)\iota^*\partial_n-e_n\cdot\slashed{D}^\prime\iota^*]
          [(n-2\lambda+2)\partial_n+x_n\Delta-e_n\cdot\slashed{D}^\prime]\\
        &=(n-2\lambda+3)[(n-2\lambda+4)\iota^*\partial_n^2-\Delta^\prime+2\slashed{D}^\prime \partial_n(e_n)]\\
        &=-(n-2\lambda+3)\slashed{D}_2(n-\lambda).
      \end{align*}   
    \end{example}
    
    Now we explain a general relationship between the families 
		$\slashed{D}_N^{BS}(\lambda)$ and $\slashed{D}_N(\lambda)$. 
    \begin{theorem}\label{BSVsSBOSpinor}
      For $N\in\N_0$ we have
      \begin{align*}
        \slashed{D}_{2N}^{BS}(\lambda)&=(-2)^N(\lambda-\frac n2-2N+\frac 12)_N(2N-1)!! \slashed{D}_{2N}(n-\lambda),\\
        \slashed{D}_{2N+1}^{BS}(\lambda)&=-(-2)^N(\lambda-\frac n2-2N-\frac 12)_N(2N+1)!! e_n\cdot \slashed{D}_{2N+1}(n-\lambda).
      \end{align*}
    \end{theorem}
    \begin{proof}
      We recall Example \ref{LowOrderScalarBSFamilySpinor}.
      Then by induction on the order we have 
      \begin{align*}
        \slashed{D}_{2N}^{BS}(\lambda)&=\slashed{D}_{2N}^{BS}(\lambda-1)\circ \slashed{P}(\lambda)\\
        &=(-1)^N 2^{N-1}(\lambda-\frac n2-2N+\frac 12)_{N-1}(2N-1)!!\slashed{D}_{2N-1}(-\lambda+n+1)\circ \slashed{P}(\lambda)\\
        &=(-1)^{N-1} 2^{N-1}(\lambda-\frac n2-2N+\frac 12)_{N-1}(2N-1)!!\big[A_0 \iota^*\partial_n^{2N}\\
        &+\sum_{k=1}^{N-1}(-1)^k A_k (\slashed{D}^\prime)^{2k}\iota^*\partial_n^{2N-2k} +A_N(\slashed{D}^\prime)^{2N}\iota^*\\
        &+B_0 \iota^*\partial_n^{2N-1}\slashed{D}^\prime(e_n\cdot)+\sum_{k=1}^{N-1}(-1)^k B_k (\slashed{D}^\prime)^{2k}\iota^*\partial_n^{2N-2k-1}\slashed{D}^\prime(e_n\cdot)    \big],
      \end{align*}
      where
      \begin{align*}
        A_0&\st  (-2\lambda+n+2N+2)(n-2\lambda+2N+1)b_0^{(N-1)}(\lambda-n-\frac 32),\\
        A_k&\st  (-2\lambda+n+2N+2)[(n-2\lambda+2N-2k+1)b_k^{(N-1)}(\lambda-n-\frac 32)\\
        &+(2N-2k+1)b_{k-1}^{(N-1)}(\lambda-n-\frac 32)]-a_{k-1}^{(N-1)}(\lambda-n-\frac 32),\\
        A_N&\st (-2\lambda+n+2N+2)-1,\\
        B_0&\st  (-2\lambda+n+2N+2)b_0^{(N-1)}(\lambda-n-\frac 32)+(n-2\lambda+2N)a_0^{(N-1)}(\lambda-n-\frac 32),\\
        B_k&\st  (-2\lambda+n+2N+2)b_k^{(N-1)}(\lambda-n-\frac 32)\\
        &+(n-2\lambda+2N-2k)a_{k}^{(N-1)}(\lambda-n-\frac 32)]+(2N-2k)a_{k-1}^{(N-1)}(\lambda-n-\frac 32).
      \end{align*}
      Then it follows that for $k=0,\ldots N$
      \begin{align*}
        A_k=(-2\lambda+n+2N+1)a_k^{(N)}(\lambda-n-\frac 12),
      \end{align*}
      while for $k=0,\ldots N-1$ it holds 
      \begin{align*}
        B_k=(-2\lambda+n+2N+1)b_k^{(N-1)}(\lambda-n-\frac 12).
      \end{align*}
      This proves the even-order case. The odd-order case is completely analogous 
			and will be omitted. 
      
    \end{proof} 
    
    \begin{remark}
      The facts analogous to Remark \ref{RemarkAboutDifferentProofs} 
			constitute different proofs of Theorem \ref{BSVsSBOSpinor}. 
    \end{remark}

%%%%%%%%%%%%%%%%%%%%%%%%%%%%%%%%%%%%%%%
  \subsection{Conformal symmetry breaking differential operators for differential forms}
        
    Conformal symmetry breaking differential operators acting on differential forms \cite{FJS,KKP}
    \begin{align}\label{eq:SBDOForms}
      D_N^{(p\to p)}(\lambda):\Omega^{p}(\R^{n})\to \Omega^{p}(\R^{n-1})
    \end{align}
    are given by
    \begin{align*}
      D_{2N}^{(p\to p)}(\lambda)&=(p-\lambda-2N)D_{2N}(\lambda)+2N(2\lambda-n+2N+1)D_{2N-1}(\lambda+1)\dm^\prime i_{e_n}\\
      &-2N D_{2N-2}(\lambda+1)\dm^\prime\delta^\prime,\\
      D_{2N+1}^{(p\to p)}(\lambda)&=(p-\lambda-2N-1)D_{2N+1}(\lambda)+D_{2N}(\lambda+1)\dm^\prime i_{e_n}-2N D_{2N-1}(\lambda+1)\dm^\prime\delta^\prime.
    \end{align*}
    Note the opposite sign convention for the family parameter $\lambda$ and a clash 
    of notation for $\Delta$ when compared to \cite{FJS}, i.e., $D^{(p\to p)}_{2N+1}(-\lambda)$ 
    and $D^{(p\to p)}_{2N+1}(-\lambda)$ introduced in \cite{FJS} correspond to 
    $(-1)^ND^{(p\to p)}_{2N}(\lambda)$ and $(-1)^ND^{(p\to p)}_{2N}(\lambda)$ 
    defined in \eqref{eq:SBDOForms}, respectively.
    As for the definition of $D_N(\lambda)$, 
    see Equation \eqref{eq:SBDOScalar}. Also note that $\iota^*$ in the definition of $D_N(\lambda)$ 
    denotes the pull-back of differential forms. The operators $D_{2N}^{(p\to p)}(\lambda)$ interpolate between 
    Branson-Gover operators for $\R^{n-1}$ and $\R^n$ equipped with the flat Euclidean metric, respectively. More precisely, it holds
    \begin{align*}
      D_{2N}^{(p\to p)}(\frac{n-1}{2}-N)&=(-1)^{N+1}\big[(\frac{n-1}{2}-p-N)(\dm^\prime\delta^\prime)^N+(\frac{n-1}{2}-p+N)(\delta^\prime\dm^\prime)^N\big],\\
      D_{2N}^{(p\to p)}(\frac n2-N)&=(-1)^{N+1}\big[(\frac{n}{2}-p-N)(\dm\delta)^N+(\frac{n}{2}-p+N)(\delta\dm)^N\big].
    \end{align*}
    Again, this appearence of Branson-Gover operators is a part of the sequence of factorizations identities for $D_{2N}^{(p\to p)}(\lambda)$, see \cite{FJS}.
        
    Let us introduce a family (termed {\it Bernstein-Sato family for differential forms of first type})
    \begin{align}\label{eq:BSFamilyDiffForms}
      D_N^{BS,(p\to p)}(\lambda)\st \iota^* P^p(\lambda-N+1)\circ\cdots \circ P^p(\lambda):\Omega^p(\R^n)\to \Omega^p(\R^{n-1}).
    \end{align}
   This definition is again inspired by \cite{C}.

   Now we state some low-order relations between $D_N^{BS,(p\to p)}(\lambda)$ and $D_N^{(p\to p)}(\lambda)$.  
    \begin{example}\label{BSFamilyLowOrderDiffForm}
      By definitions \eqref{eq:SBDOForms} and  \eqref{eq:SBDOScalar} we have 
      \begin{align*}
        D_1^{(p\to p)}(\lambda)&=(p-\lambda-1)\iota^*\partial_n+\dm^\prime\iota^* i_{e_n},\\
        D_2^{(p\to p)}(\lambda)&=(p-\lambda-2)\Delta^\prime\iota^*+(2\lambda-n+3)(p-\lambda-2)\iota^*\partial_n^2\\
        &+2(2\lambda-n+3)\dm^\prime\iota^* i_{e_n}\partial_n-2\dm^\prime\delta^\prime\iota^*,\\
        D_3^{(p\to p)}(\lambda)&=(p-\lambda-3)\Delta^\prime\iota^*\partial_n
          +\frac 13 (2\lambda-n+5)(p-\lambda-3)\iota^*\partial_n^3\\
          &+(2\lambda-n+5)\dm^\prime\iota^*i_{e_n}\partial_n^2+\Delta^\prime\dm^\prime\iota^* i_{e_n}-2\dm^\prime\delta^\prime\iota^*\partial_n.
      \end{align*}
      It is then straightforward to verify
      \begin{align*}
        D_1^{BS,(p\to p)}(\lambda)&=-(2\lambda-n-2)(\lambda-p)D_1^{(p\to p)}(n-\lambda),\\
        D_2^{BS,(p\to p)}(\lambda)&=-(2\lambda-n-4)(\lambda-n+p-1)(\lambda-p-1)_2D_2^{(p\to p)}(n-\lambda),\\
        D_3^{BS,(p\to p)}(\lambda)&=3(2\lambda-n-6)(2\lambda-n-4)(\lambda-p-2)_3(\lambda-n+p-2)_2 D^{(p\to p)}_3(n-\lambda).
      \end{align*}
    \end{example}

    Now we state a general relationship between the families 
		$D_N^{BS,(p\to p)}(\lambda)$ and $D^{(p\to p)}_N(\lambda)$. 
    \begin{theorem}\label{BSVsSBODiffForms}
      For $N\in\N$ holds
      \begin{align*}
        D_{2N}^{BS,(p\to p)}(\lambda)&=(-2)^N(\lambda-\frac n2-2N)_N(2N-1)!!\times\\
        &\times(\lambda-n+p-2N+1)_{2N-1}(\lambda-p-2N+1)_{2N}D^{(p\to p)}_{2N}(n-\lambda),\\
        D_{2N+1}^{BS,(p\to p)}(\lambda)&=(-2)^{N+1}(\lambda-\frac n2-2N-1)_{N+1}(2N+1)!!\times\\
        &\times(\lambda-n+p-2N)_{2N}(\lambda-p-2N)_{2N+1}D^{(p\to p)}_{2N+1}(n-\lambda).
      \end{align*}
    \end{theorem}
    \begin{proof}
      The proof goes by induction and starts with Example \ref{BSFamilyLowOrderDiffForm}. 
      By definition \eqref{eq:BSFamilyDiffForms} and induction hypothesis it follows  
      \begin{align*}
        D_{2N}^{BS,(p\to p)}(\lambda)&=D_{2N-1}^{BS,(p\to p)}(\lambda-1)\circ P^p(\lambda)\\
        &=c(2N-1,\lambda-1)D_{2N-1}^{(p\to p)}(n-\lambda+1)\circ P^p(\lambda)
      \end{align*}
      for 
      \begin{multline*}
        c(2N-1,\lambda-1)=(-2)^N(\lambda-\frac n2-2N)_N(2N-1)!!\times\\
        \times(\lambda-p-2N+1)_{2N-1}(\lambda-n+p-2N+1)_{2N-2}.
      \end{multline*}
      Now \eqref{eq:SBDOForms} gives that 
      \begin{align}
        D_{2N-1}^{(p\to p)}(n-\lambda+1)\circ P^p(\lambda)=\big[&(\lambda-n+p-2N) D_{2N-1}(n-\lambda+1)\circ P^p(\lambda)\notag\\
        &+D_{2N-2}(n-\lambda+2)\dm^\prime i_{e_n} P^p(\lambda)\notag\\
        &-(2N-2) D_{2N-3}(n-\lambda+2)\dm^\prime\delta^\prime P^p(\lambda) \big].\label{eq:help3}
      \end{align}
      The individual summands on the right hand side of the last display 
      simplify by \eqref{eq:BSOperator1DiffForm} and the identities
      \begin{align*}
        \iota^*\partial_n^k(x_n F)&=k \iota^*\partial_n^{k-1}F,\quad \iota^*\dm\delta=\dm^\prime\delta^\prime\iota^*-\dm^\prime\iota^* i_{e_n}\partial_n, 
          \quad \iota^* \varepsilon_{e_n}\delta=0,\quad \iota^* \dm i_{e_n}=\dm^\prime \iota^* i_{e_n},
      \end{align*}
      where $k\in\N$ and $F$ is a differential operator.
      Hence we see by Corollary \ref{RecurrenceForJuhl}  
      \begin{align*}
        D_{2N-1}(n-\lambda+1) &P^p(\lambda)=(\lambda-n+p-1)(\lambda-p)\sum_{k=0}^{N-1}a_k^{(N-1)}(\lambda-n)(\Delta^\prime)^k\iota^*\partial_n^{2N-2k}\\
        +\sum_{k=0}^{N-1}&\big[(n-2p)(2N-2k-1)-(2\lambda-n-2)(\lambda-p)\big]b_k^{(N-1)}(\lambda-n-1)\times \\
        &\times (\Delta^\prime)^k\iota^*\partial_n^{2N-2k-1}\dm^\prime i_{e_n}\\
        -(n-2p)\sum_{k=0}^{N-1}&(2N-2k-1)b_k^{(N-1)}(\lambda-n-1)(\Delta^\prime)^k\iota^*\partial_n^{2N-2k-2}\dm^\prime\delta^\prime .
      \end{align*}
      Similarly, the identities 
      \begin{align*}
        P(\lambda)&=P(\lambda-1)-2\partial_n,\quad \iota^*\dm^\prime i_{e_n}\dm\delta=\dm^\prime\delta^\prime\iota^*\partial_n-\dm^\prime \iota^* i_{e_n}\partial_n^2,\\
        \iota^*\dm^\prime i_{e_n} \varepsilon_{e_n}\delta&=\dm^\prime\delta^\prime\iota^*-\dm^\prime\iota^* i_{e_n}\partial_n,\quad \iota^*\dm^\prime i_{e_n}\dm i_{e_n}=\dm^\prime\iota^* i_{e_n}\partial_n
      \end{align*}
      and Corollary \ref{RecurrenceForJuhl} allow to conclude 
      \begin{align*}
        D_{2N-2}&(n-\lambda+2)\dm^\prime i_{e_n} P^p(\lambda)\\
        =\sum_{k=0}^{N-1}\bigg[&-(2N-1)(2\lambda-n-2N-2)(\lambda-n+p-1)(\lambda-p)b_k^{(N-1)}(\lambda-n-1)\\
        &+\big[-2(\lambda-n+p-1)(\lambda-p) +(n-2p)(2N-2k-2)\\
        &+(2\lambda-n-2)(\lambda-n+p)-(2\lambda-n-2)(\lambda-p)\big]a_k^{(N-1)}(\lambda-n-2)\bigg]\times\\
        &\times (\Delta^\prime)^k\iota^* \partial_n^{2N-2k-1}\dm^\prime i_{e_n}\\
        +\sum_{k=0}^{N-1}\big[&-(n-2p)(2N-2k-2)-(2\lambda-n-2)(\lambda-n+p)\big]a_k^{(N-1)}(\lambda-n-2)\times\\
        &\times(\Delta^\prime)^k\iota^* \partial_n^{2N-2k-2}\dm^\prime\delta^\prime . \\
      \end{align*}
      Finally, by 
      \begin{align*}
        \dm^\prime\delta^\prime\iota^*\dm\delta=(\dm^\prime\delta^\prime)^2\iota^*-\dm^\prime\delta^\prime\dm^\prime\iota^* i_{e_n}\partial_n,\quad 
        \dm^\prime\delta^\prime\iota^*(\varepsilon_{e_n}\delta)=0,\quad \dm^\prime\delta^\prime\iota^* \dm i_{e_n}=\dm^\prime\delta^\prime\dm^\prime\iota^*i_{e_n}
      \end{align*}
      and Corollary \ref{RecurrenceForJuhl} we have 
      \begin{align*}
        D_{2N-2}&(n-\lambda+2)\dm^\prime \delta^\prime P^p(\lambda)\\
        =\sum_{k=0}^{N-1}&\big[-(n-2p)(2N-2k-1)+(2\lambda-n-2)(\lambda-p)  \big]b_{k-1}^{(N-2)}(\lambda-n-2)\times\\
        &\times (\Delta^\prime)^k\iota^*\partial_n^{2N-2k-1}\dm^\prime i_{e_n}\\
        +\sum_{k=0}^{N-1}&\bigg[(\lambda-n+p-1)(\lambda-p)\big[a_k^{(N-1)}(\lambda-n-1)-2b_k^{(N-2)}(\lambda-n-2)  \big]  \\
        &+(n-2p)(2N-2k-1)b_{k-1}^{(N-2)}(\lambda-n-2)\bigg](\Delta^\prime)^k\iota^*\partial_n^{2N-2k-2}\dm^\prime\delta^\prime.
      \end{align*}
      Consequently, we see that Equation \eqref{eq:help3} simplifies to 
      \begin{align*}
        D_{2N-1}&^{(p\to p)}(n-\lambda+1)\circ P^p(\lambda)\\
        =&(\lambda-n+p-1)(\lambda-p)(\lambda-n+p-2N)D_{2N}(n-\lambda)\\
        &-(\lambda-n+p-1)(\lambda-p)(2N)(2\lambda-n-2N-1)D_{2N-1}(n-\lambda+1)\dm^\prime i_{e_n}\\
        &-(\lambda-n+p-1)(\lambda-p)(2N-2)D_{2N-2}(n-\lambda+1)\dm^\prime\delta^\prime\\
        =&(\lambda-n+p-1)(\lambda-p)D_{2N}^{(p\to p)}(n-\lambda)
      \end{align*}
      and hence we have 
      \begin{align*}
        D_{2N}^{BS,(p\to p)}(n-\lambda)&=(-2)^N(\lambda-\frac n2-2N)_N(2N-1)!!\times\\
        &\times(\lambda-n+p-2N+1)_{2N-1}(\lambda-p-2N+1)_{2N}D^{(p\to p)}_{2N}(n-\lambda).
      \end{align*}
      The odd-order families follow by a similar argument. The proof is complete.
    \end{proof}
    An immediate consequence of the last theorem is
    \begin{corollary}
      Let $N\in\N_0$. Then it holds
      \begin{align*}
         D_{2N-1}^{(p\to p)}(n-\lambda+1)\circ P^p(\lambda)=&(\lambda-n+p-1)(\lambda-p)D_{2N}^{(p\to p)}(n-\lambda),\\
         D_{2N}^{(p\to p)}(n-\lambda+1)\circ P^p(\lambda)=&-(\lambda-n+p-1)(\lambda-p)(2N+1)\times\\
         &\times(2\lambda-n-2N-2)D_{2N+1}^{(p\to p)}(n-\lambda).
      \end{align*}
    \end{corollary}
    
    \begin{remark}
      The Bernstein-Sato families of the first type 
			$D_{N}^{BS,(p\to p)}(\lambda), N\in\N_0,$ induce 
      the full classification list for 
      conformal symmetry breaking differential operators on differential 
      forms, cf. \cite[Theorem $3$]{FJS}. For example, the Bernstein-Sato families 
			of the second type arise by post- and pre-composition 
      of $D_{N}^{BS,(p\to p)}(\lambda), N\in\N_0$, with the Hodge-star operators on 
			$\R^{n-1}$ and $\R^n$, respectively. 
    \end{remark}

%%%%%%%%%%%%%%%%%%%%%%%%%%%%%%%%%%%%%%%%%%%%%%%%%%%%%%%
%%%%
%%%%%%%%%%%%%%%%%%%%%%%%%%%%%%%%%%%%%%%%%%%%%%%%%%%%%%%

\end{document}